\documentclass{amsart}

\usepackage[letterpaper,margin=1in]{geometry}
\usepackage[foot]{amsaddr}

\usepackage{amsmath,amssymb,amsthm}
\usepackage{hyperref}
\usepackage{mathtools}
\usepackage{tikz-cd}
\usepackage{enumerate}
\usepackage[square,numbers]{natbib}

\newtheorem{theorem}{Theorem}[section]
\newtheorem{proposition}[theorem]{Proposition}
\newtheorem{lemma}[theorem]{Lemma}

\theoremstyle{definition}
\newtheorem{definition}[theorem]{Definition}
\newtheorem{construction}[theorem]{Construction}

\theoremstyle{remark}
\newtheorem{remark}[theorem]{Remark}

\numberwithin{equation}{section}

\newcommand{\R}{\mathbb{R}}
\newcommand\C{\mathbb{C}}
\newcommand\N{\mathbb{N}}
\renewcommand\H{\mathbb{H}}

\newcommand{\unicodecurlyB}{ℬ}

\DeclareMathOperator{\Spec}{Spec}

\DeclareMathOperator{\im}{im}
\DeclareMathOperator{\id}{id}
\DeclareMathOperator{\alg}{alg}

\DeclarePairedDelimiter{\ip}{\langle}{\rangle}
\DeclarePairedDelimiter{\norm}{\lVert}{\rVert}

\def\semicolon{;}
\def\applytolist#1{
	\expandafter\def\csname multi#1\endcsname##1{
		\def\multiack{##1}\ifx\multiack\semicolon
		\def\next{\relax}
		\else
		\csname #1\endcsname{##1}
		\def\next{\csname multi#1\endcsname}
		\fi
		\next}
	\csname multi#1\endcsname}

\def\calc#1{\expandafter\def\csname c#1\endcsname{{\mathcal #1}}}
\applytolist{calc}QWERTYUIOPLKJHGFDSAZXCVBNM;
\def\bbc#1{\expandafter\def\csname b#1\endcsname{{\mathbb #1}}}
\applytolist{bbc}QWERTYUIOPLKJHGFDSAZXCVBNM;
\def\fc#1{\expandafter\def\csname f#1\endcsname{{\mathfrak #1}}}
\applytolist{fc}QWERTYUIOPLKJHGFDSAZXCVBNM;
\def\rc#1{\expandafter\def\csname r#1\endcsname{{\mathrm #1}}}
\applytolist{rc}QWERTYUIOPLKJHGFDSAZXCVBNM;

\title[Models and subordination for operator-valued free convolution powers]{Operator models and analytic subordination for operator-valued free convolution powers}

\begin{document}
	
	\author{Ian Charlesworth$^\circ$}
	\address{$^\circ$School of Mathematics, Cardiff University, United Kingdom}
	\email{\href{mailto:charlesworthi@cardiff.ac.uk}{charlesworthi@cardiff.ac.uk}}
	
	\author{David Jekel$^\bullet$}
	\address{$^\bullet$Department of Mathematical Sciences, University of Copenhagen, Denmark}
	\email{\href{mailto:daj@math.ku.dk}{daj@math.ku.dk}}
	\urladdr{https://davidjekel.com}

	\begin{abstract}
		We revisit the theory of operator-valued free convolution powers given by a completely positive map $\eta$.  We first give a general result, with a new analytic proof, that the $\eta$-convolution power of the law of $X$ is realized by $V^*XV$ for any operator $V$ satisfying certain conditions, which unifies Nica and Speicher's construction in the scalar-valued setting and Shlyakhtenko's construction in the operator-valued setting.  Second, we provide an analog, for the setting of $\eta$-valued convolution powers, of the analytic subordination for conditional expectations that holds for additive free convolution.  Finally, we describe a Hilbert-space manipulation that explains the equivalence between the $n$-fold additive free convolution and the convolution power with respect to $\eta = n \id$.
	\end{abstract}
	\maketitle
	
	\section{Introduction}
	
	The classical convolution of probability measures describes the distribution of the sum of independent variables realizing the measures.
	For a probability measure $\mu$, one can therefore make sense of $\mu^{\ast n}$ for $n \in \N$ as the sum of $n$ independent random variables with identical distribution $\mu$.
	Since convolution of probability measures corresponds to multiplication of their Fourier transforms, there is an immediate candidate to define as $\mu^{\ast t}$ for $t > 0$ by taking an appropriate power on the Fourier transform side; however, even if one begins with a probability measure there is not in general a guarantee that the output of this operation is a positive measure of total mass $1$.
	This does happen, however, for the so-call \emph{infinitely divisible} probability measures: for such a measure $\mu$ one has a L\'evy process consisting of measures $(\mu^{\ast t})_{t > 0}$ which in particular is a semi-group with respect to convolution.
	
	In the non-commutative setting of free probability, one likewise has a notion of free convolution: the law $\mu \boxplus \nu$ is the law of the sum of freely independent random variables having laws $\mu$ and $\nu$ individually.
	Corresponding to the Fourier transform of measures is Voiculescu's $R$-transform, which associates to a probability measure $\mu$ an analytic function $R_\mu$ and linearizes free convolution: $R_{\mu \boxplus \nu} = R_\mu + R_\nu$.
	(Strictly speaking, $R$ is therefore analogous to the logarithm of the Fourier transform.)
	One arrives at the notion of $\boxplus$-infinitely divisible law $\mu$, namely a law for which $t R_\mu$ is the $R$-transform of a probability measure for every $t > 0$, and a corresponding free L\'evy process.
	The central limit distributions in these settings are prime examples of infinite divisibility: if $\gamma_t$ is the centered Gaussian distribution of variance $t$, one has $\gamma_t = \gamma_1^{*t}$; similarly, if $\sigma_t$ is the centered semi-circular distribution of variance $t$, one has $\sigma_t = \sigma_1^{\boxplus t}$.
	These form semi-groups with respect to $*$ and $\boxplus$, respectively.  A result of Bercovici and Pata shows that there is a deep connection between the $\boxplus$-infinitely divisible laws and the $*$-infinitely divisible laws: there is a bijection between these two sets of laws which preserves domains of partial attraction \cite{BP1999}.
	
	Surprisingly, for every probability measure $\mu$ and every $t \geq 1$, the function $tR_\mu$ is the $R$-transform of some probability measure $\mu^{\boxplus t}$, which is called the $t^{\text{th}}$ \emph{free convolution power of $\mu$}; this is in stark contrast to the classical setting where real-valued convolution powers often do not exist.
	This fact was first shown for sufficiently large $t$ by Bercovici and Voiculescu \cite{BV1995}, and then proved for all $t \geq 1$ by Nica and Speicher \cite{NS1996}.
	Moreover, Nica and Speicher were able to produce an explicit model of these distributions: if $x = x^*$ is an element of a tracial non-commutative probability space $(\cA, \phi)$ with law $\mu$ and $p \in \cA$ is a projection freely independent from $x$ with $\phi(p) = 1/t$, then $pxp$ has law $\mu^{\boxplus t}$ in the non-commutative probability space $(p\cA p, t \phi|_{p\cA p})$.
	
	This paper studies free convolution powers in the setting of operator-valued free probability, where the algebra of scalars $\bC$ is replaced by a $\rC^*$-algebra $\cB$ and the state by a conditional expectation onto $\cB$.
	In this setting, the central limit distribution -- the semi-circular law -- is still characterized by its first and second moment, but the second moment is now a completely positive map $\eta : \cB \to \cB$ \cite{Voiculescu1995}.
	Motivated by scalar-valued free convolution powers, Anshelevich, Belinschi, Fevrier, and Nica showed in \cite{ABFN2013} that there is a natural operation $\mu \mapsto \mu^{\boxplus \eta}$ with the desirable property that if $\sigma$ is a $\cB$-valued semicircular distribution with variance $\id_\cB$ then $\sigma^{\boxplus \eta}$ is the $\cB$-valued semicircular distribution with variance $\eta$.
	Their convolution is defined for arbitrary $\cB$-valued distributions $\mu$ when $\eta - \id_\cB$ is completely positive (mimicking the scalar-valued condition that $t \geq 1$), and for certain distributions for arbitrary $\eta$ completely positive.
	It was later shown by Shlyakhtenko \cite{Shlyakhtenko2013} that the condition $\eta - \id_\cB$ completely positive is necessary to ensure that the $\eta$-convolution $\cdot^{\boxplus \eta}$ preserves positivity for all (multivariate) distributions.
	The precise definition of $\eta$-convolution is given at the start of Section~\ref{sec:opconv}, but it essentially amounts to a condition on $\cB$-valued $R$-transforms: that
	\[ R_{\mu^{\boxplus\eta}} = \eta\circ R_\mu \]
	and a similar condition for matrix amplifications, on an appropriate domain.
	(In the scalar-valued case, the completely positive maps $\bC \to \bC$ are just multiplication by elements of $[0, \infty)$, and so this extends the usual notion of free convolution powers.)

	Shlyakhtenko exhibited in \cite{Shlyakhtenko2013} a construction of $\eta$-convolution powers: given $\cB$ and $\eta : \cB \to \cB$ such that $\eta - \id_{\cB}$ is completely positive, he produces an operator $V$ living on a free Fock space so that whenever $\mu$ is a $\cB$-valued distribution realized by a variable $X$ free from $V$ (in a potentially larger space), $V^* X V$ has distribution $\mu^{\boxplus\eta}$.
	
	Our first main theorem, Theorem~\ref{thm: realization}, gives conditions for an operator $V$ which guarantee that, when $V$ is free from $X$ which has distribution $\mu$, the distribution of $V^*XV$ is $\mu^{\boxplus\eta}$.
	Shlyakhtenko's construction satisfies the conditions of our theorem, but the properties we demand are weaker than those needed for the proof used in \cite{Shlyakhtenko2013}.
	In particular, this allows us to construct a suitable operator $V$ living on $\cB\oplus \cB\otimes_{\eta-\id}\cB$ (see Lemma~\ref{lem:amodel} and the preceding discussion), and to make the connection to the scalar case more explicit.  The main stroke of our proof is as follows: writing $\nu$ for the distribution of $V^*XV$, we show that $\eta\circ R_\mu$ satisfies the functional equation defining $R_{\nu}$ in terms of its Cauchy transform $G_\nu$, implying that $\nu = \mu^{\boxplus \eta}$.  This approach is inspired by the proof of the additivity of the (scalar) $R$-transform over $\boxplus$ given by Haagerup \cite[\S 3]{Haagerup1997} and Lehner \cite[Theorem 3.1]{Lehner2001}, which was also adapted to the operator-valued case by Dykema \cite[Lemma 4.5]{Dykema2006}.
	
	Our second main result Theorem \ref{thm: subordination} gives a formula relating analytic subordination and conditional expectations in the setting of operator-valued convolution powers.
	Analytic subordination of Cauchy transforms has been an important tool for studying additive free convolution, beginning with the works of Voiculescu \cite{Voiculescu1993} and Biane \cite{Biane1998}.
	Given probability measures $\mu, \nu$ on $\R$ there are analytic functions $\omega_\mu, \omega_\nu$ such that
	\[G_\mu\circ \omega_\mu(z) = G_\nu\circ\omega_\nu(z) = G_{\mu\boxplus\nu}(z).\]
	The theory was further developed and codified by Voiculescu in \cite{Voiculescu2002b}.
	A key point in this theory is that for free variables $X$ and $Y$, the conditional expectation of a resolvent in $X+Y$ onto the von Neumann algebra generated by $X$ is a resolvent in $X$, namely, $E_{\mathrm{W}^*(X)}[(z - X+Y)^{-1}] = (\omega_\mu(z) - X)^{-1}$. 
	The theory of analytic subordination for additive free convolution was adapted to the $\mathcal{B}$-valued setting in \cite{Voiculescu2002b,BMS2017,Liu2021,BB2022}.  Our Theorem \ref{thm: subordination} gives that an analogous conditional expectation formula holds in the setting of operator-valued convolution powers with respect to a completely positive map $\eta$, even though this cannot always be realized as a free sum.
	Namely, if $X$ and $V$ are as in Theorem \ref{thm: realization} and $\mathcal{E}_1$ is the conditional expectation onto the algebra generated by $X$, then
	\[
	\mathcal{E}_1[V(z - V^*XV)^{-1}V^*] = (F(z) - X)^{-1},
	\]
	where $F = G_\mu^{-1} \circ G_{\mu^{\boxplus \eta}}$ is the $\mathcal{B}$-valued subordination function.
	
	As a final point of interest, we examine the particular case of $n$-fold convolution with $n \in \N$ in \S \ref{sec: convolution sums}.  The distribution of the $n$-fold free convolution can be modeled in two ways, either by the sum of $n$ copies of the original variable that are freely independent over $\cB$, or as a convolution power corresponding to the completely positive map $\eta = n \id$.  We give an explicit transformation of the $\cB$-$\cB$-correspondences that relates these two models, hence obtaining a new heuristic for why these operations should be the same.

	The structure of this paper is as follows.
	In Section~\ref{sec:prelims} we recall much of the background material needed for the rest of the paper and establish notation.
	Following that, Section~\ref{sec:opconv} contains our realization of operator-valued convolution powers.
	Section~\ref{sec:condexp} shows that our model can be used to produce subordination functions in a manner analogous to the scalar-valued setting, satisfying the expected conditional expectation relation.
	Finally, in Section~\ref{sec: convolution sums} we investigate the particular case where $\eta = n\operatorname{id}$ and show that an isomorphism of the underlying $\cB$-$\cB$-correspondence produces the familiar model of $n$ freely independent variables being summed.
	
	\subsection*{Acknowledgements}
	
	We thank Roland Speicher for discussions early in the development of this work.  We thank the Fields Institute for their hospitality during the Fall 2023 thematic program on operator algebras that supported our collaboration.  We thank the Mathematische Forschungsinstitut Oberwolfach for hosting us during the workshop ``Non-commutative Function Theory and Free Probability'' in Spring 2024 where these results were first presented.  DJ was partially supported by the Fields Institute, the National Science Foundation (US; grant DMS-2002826), the Natural Sciences and Engineering Research Council (Canada; grant RGPIN-2017-05650), the Danmarks Frie Forskningsfond (Denmark; grant 1026-00371B), and the Horizon Marie-Sk{\l}odowska--Curie Actions (FRANCO, grant 101209517).  IC received travel support from the Fields Institute.  We thank the referees for their helpful comments.
	
	\section{Preliminaries}
	\label{sec:prelims}
	
	In this section, we review the necessary background on $\cB$-valued non-commutative probability theory.  For general background on free probability, see \cite[\S 5]{AGZ2009}, \cite{MS2017}, and for the operator-valued setting, see \cite{Speicher1998}, \cite{BMS2017}, \cite[\S 2-5]{JekelThesis}.
	
	\subsection{\texorpdfstring{$\cB$}{\unicodecurlyB}-valued non-commutative probability spaces} \label{sec: B prob space}
	
	In non-commutative probability theory, one replaces the (commutative) algebra $L^\infty$ of a probability space with a general unital $\mathrm{C}^*$ or von Neumann algebra $\cA$.  The expectation is then replaced by a state $\phi: \cA \to \C$.
	Often $\cA$ is represented concretely as bounded operators on a Hilbert space $H$, and the state $\phi$ is given by a vector $\xi$, i.e.\ $\phi(a) = \ip{\xi, a \xi}$.
	Even if $(\cA,\phi)$ is given abstractly without a particular representation on a Hilbert space, a Hilbert space $H_\phi$ can be obtained by the GNS construction equipping $\cA$ with the sesquilinear form $\ip{a,b}_\phi = \phi(a^*b)$, and there is a representation $\pi_\phi: A \to B(H_\phi)$ given by left multiplication, such that $\phi(a) = \ip{1,\pi_\phi(a) 1}_{H_\phi}$.
	
	We work in the setting of $\cB$-valued non-commutative probability, that is, we are given a $\mathrm{C}^*$-algebra $\cB$, and all the scalar-valued quantities such as the expectation and the inner product on a Hilbert space are replaced by $\cB$-valued versions.  We assume familiarity with the basic theory of unital $\mathrm{C}^*$-algebras, matrices over a $\mathrm{C}^*$-algebra, and completely positive maps.  We refer to \cite[Chapter II]{Blackadar2006} for a summary of results and references.
	
	\begin{definition}[$\cB$-valued probability space]
		Let $\cB$ be a unital $\mathrm{C}^*$-algebra.  A \emph{$\cB$-valued probability space} is a pair $(\cA \supseteq \cB, E: \cA \to \cB)$ where $\cA$ is a unital $\mathrm{C}^*$-algebra and $\cB \subseteq \cA$ is a unital inclusion (injective $*$-homomorphism), and $E: \cA \to \cB$ is a \emph{conditional expectation}, that is,
		\begin{itemize}
			\item $E|_{\cB} = \id_{\cB}$.
			\item $E$ is a $\cB$-$\cB$-bimodule map, so $E[b_1ab_2] = b_1 E[a] b_2$ for $a \in \mathcal{A}$ and $b_1, b_2 \in \mathcal{B}$.
			\item $E$ is completely positive, that is, if $n \geq 1$ and $A \in M_n(\cA)$, then $E^{(n)}[A^*A] \geq 0$ in $M_n(\cB)$, where $E^{(n)}: M_n(\cA) \to M_n(\cB)$ is the entrywise application of $E$.
		\end{itemize}
	\end{definition}
	
	Next, we recall $\mathrm{C}^*$-correspondences, the $\cB$-valued analog of representations on Hilbert spaces; this will be important in \S \ref{sec: convolution sums} where we perform ``$\cB$-valued Hilbert space'' manipulations.
	For background on correspondences, see \cite{Paschke1973}, \cite{Lance1995}, \cite[\S II.7.1 - II.7.2]{Blackadar2006}, and the references therein.  We start with right Hilbert $\mathcal{B}$-modules, which are an analog of Hilbert spaces with a $\cB$-valued right $\cB$-linear inner product.
	
	\begin{definition}
		Let $\mathcal{B}$ be a unital $\mathrm{C}^*$-algebra.  If $\mathcal{H}$ is a right $\mathcal{B}$-module, then a \emph{$\mathcal{B}$-valued semi-inner product} is a map $\ip{\cdot,\cdot}: \mathcal{H} \times \mathcal{H} \to \mathcal{B}$ such that for $h_1$, $h_2 \in \mathcal{H}$:
		\begin{enumerate}[(1)]
			\item $h_2 \mapsto \ip{h_1,h_2}$ is a right $\mathcal{B}$-module map;
			\item $\ip{h_2,h_1} = \ip{h_1,h_2}^*$; and
			\item $\ip{h_1,h_1} \geq 0$.
		\end{enumerate}
	\end{definition}
	
	One can show that the semi-inner product must satisfy an analogue of the Cauchy--Schwarz inequality and hence $\norm{h} := \norm{\ip{h,h}}_{\mathcal{B}}^{1/2}$ defines a semi-norm on $\mathcal{H}$.  We also have $\norm{hb} \leq \norm{h} \norm{b}_{\cB}$ for $h \in \mathcal{H}$ and $b \in \mathcal{B}$.
	
	\begin{definition}
		If this semi-norm is in fact a norm which makes $\mathcal{H}$ a Banach space, then we say that $\mathcal{H}$ is a \emph{right Hilbert $\mathcal{B}$-module}.
		In general, if $\mathcal{H}$ has a $\mathcal{B}$-valued semi-inner product, then the completion of $\mathcal{H} / \{h: \norm{h} = 0\}$ is a right Hilbert $\mathcal{B}$-module with the right $\mathcal{B}$-action and the $\mathcal{B}$-valued inner product induced in the natural way from those of $\mathcal{H}$.  We refer to this module as the \emph{separation-completion} of $\mathcal{H}$ with respect to $\ip{\cdot,\cdot}$.
	\end{definition}
	
	\begin{definition}
		Let $\mathcal{H}_1$ and $\mathcal{H}_2$ be Hilbert $\mathcal{B}$-modules. We say that a linear map $T: \mathcal{H}_1 \to \mathcal{H}_2$ is \emph{right $\mathcal{B}$-modular} if $T(hb) = (Th)b$ for $h \in \mathcal{H}_1$ and $b \in \mathcal{B}$.  We say that $T$ is \emph{adjointable} if there exists a map $T^*: \mathcal{H}_2 \to \mathcal{H}_1$ such that
		\[
		\ip{Th_1,h_2} = \ip{h_1,T^*h_2} \text{ for all } h_1 \in \mathcal{H}_1 \text{ and } h_2 \in \mathcal{H}_2.
		\]
		We denote by $\mathcal{L}(\mathcal{H})$ the space of bounded, right $\mathcal{B}$-modular, adjointable operators on a right Hilbert $\mathcal{B}$-module $\mathcal{H}$.
		When we want to emphasize the right $\cB$-module structure of $\cH$, we may write $\cH_\cB$ and $\cL(\cH_\cB)$.
	\end{definition}
	
	Unlike the case of scalar-valued Hilbert spaces, adjointability is \emph{not} automatic.  One can check that $\mathcal{L}(\mathcal{H})$ is a $\mathrm{C}^*$-algebra (see for instance \cite[p.\ 8]{Lance1995}).  Now we can define the $\cB$-valued analog of representations on a Hilbert space.
	
	\begin{definition} \label{def: correspondence}
		Let $\cA$ and $\cB$ be $\mathrm{C}^*$-algebras.  An \emph{$\cA$-$\cB$-correspondence} is a right Hilbert $\cB$-module $\mathcal{H}$ together with a $*$-homomorphism $\pi: \mathcal{A} \to \mathcal{L}(\mathcal{H})$.
		Such a representation endows $\mathcal{H}$ with the structure of an $\mathcal{A}$-$\mathcal{B}$-bimodule, and so we will write $a \xi b$ for $\pi(a)[\xi b]$ where $\xi \in \mathcal{H}$, $a \in \cA$, and $b \in \cB$.
	\end{definition}
	
	A completely positive map $\psi: \cA \to \cB$ naturally gives rise to an $\cA$-$\cB$-correspondence $\cA \otimes_{\psi} \cB$ as follows.
	
	\begin{lemma}[{\cite[Theorem 5.2]{Paschke1973}, see also \cite[Theorem 5.6]{Lance1995}}] \label{lem: tensor product over CP map}
		Let $\cA$ and $\cB$ be $\mathrm{C}^*$-algebras and $\psi: \cA \to \cB$ be completely positive.  On the algebraic tensor product $\cA \otimes_{\alg} \cB$, define a $\cB$-valued sesquilinear map by
		\[
		\ip{a \otimes b, a' \otimes b'}_\psi = b^* \psi(a^*a') b'.
		\]
		Then $\ip{\cdot,\cdot}_\psi$ is a $\cB$-valued semi-inner product.  The associated separation-completion $\cA \otimes_\psi \cB$ has an $\cA$-$\cB$-correspondence structure where the left action of $\cA$ is given on simple tensors by $a (a' \otimes b') = aa' \otimes b'$.
	\end{lemma}
	
	Next, we discuss the relationship between $\cB$-valued correspondences and $\cB$-valued probability spaces.  Suppose $\cB \subseteq \cA$ and $\mathcal{H}$ is an $\cA$-$\cB$-correspondence.  If $\xi \in \mathcal{H}$, then under what conditions is the map $a \mapsto \ip{\xi, a \xi}$ a conditional expectation?
	The following is proved in 
	\cite[Lemma 2.10]{JekelLiu2020}, and part of it was given also in \cite[Proposition 2.10]{Liu2021}.
	
	\begin{lemma} \label{lem:unitvector}
		Let $\mathcal{B} \subseteq \mathcal{A}$ be a unital inclusion of unital $\mathrm{C}^*$-algebras.  Let $\mathcal{H}$ be a $\mathcal{A}$-$\mathcal{B}$-correspondence and $\xi \in \mathcal{H}$ and define $E: \mathcal{A} \to \mathcal{B}$ by $E(a) := \ip{\xi, a \xi}$.  Then the following are equivalent:
		\begin{enumerate}[(1)]
			\item $E$ is a $\mathcal{B}$-valued conditional expectation.
			\item $\ip{\xi, b\xi} = b$ for every $b \in \mathcal{B}$.
			\item $\ip{\xi,\xi} = 1$ and $b \xi = \xi b$ for every $b \in \mathcal{B}$. 
		\end{enumerate}
	\end{lemma}
	
	\begin{remark} \label{rem: conditional expectation automatic}
		This lemma is closely related to the well-known fact that if $\cB \subseteq \cA$ and $E: \cA \to \cB$ is completely positive with $E|_{\cB} = \id$, then $E$ is automatically a conditional expectation.
		Indeed, letting $\xi = 1 \otimes 1$ in $\cA \otimes_E \cB$, we see that Lemma \ref{lem:unitvector} (2) is equivalent to $E|_{\cB} = \id$, and so $E$ is a conditional expectation.
	\end{remark}
	
	\begin{definition}
		In the setting of Lemma~\ref{lem:unitvector}, if $\xi$ satisfies (1) - (3), then we call $\xi$ a \emph{$\cB$-central unit vector} and we call $(\cH,\xi)$ a \emph{pointed $\cA$-$\cB$-correspondence}.
	\end{definition}
	
	Hence, for $\cB \subseteq \cA$, every pointed $\cA$-$\cB$-correspondence makes $\cA$ into a $\cB$-valued probability space.  Conversely, given a $\cB$-valued probability space $(\cA,E)$, one can construct a canonical $\cA$-$\cB$-correspondence $L^2(\cA,E)$ by the operator-valued GNS construction as follows.  Note that $\mathcal{A}$ is a right $\mathcal{B}$-module and we can define a $\mathcal{B}$-valued semi-inner product on $\mathcal{A}$ by $\ip{a_1,a_2} = E[a_1^*a_2]$.  We denote the separation-completion with respect to this inner product by $L^2(\mathcal{A},E)$.  One can check that $L^2(\mathcal{A},E)$ is a $\mathcal{A}$-$\mathcal{B}$-correspondence.  If we denote by $\xi$ the equivalence class of the vector $1 \in \mathcal{A}$, then we have
	\[
	E[a] = \ip{\xi, a \xi}.
	\]
	Thus, every $\cB$-valued non-commutative probability space can be realized using some pointed $\cA$-$\cB$-correspondence.
	
	In some contexts of $\cB$-valued probability spaces, one may require an additional condition that the $*$-homomorphism $\pi: \cA \to \mathcal{L}(\cH_{\cB})$ is injective, i.e., the representation of $\cA$ is faithful.
	This is equivalent to saying that $E[a_1 a a_2] = 0$ for all $a_1$, $a_2 \in \cA$ only if $a = 0$ (since elements from $\cA$ are dense in $L^2(\cA,E)$).
	Generally, if $(\cH,\xi)$ is a pointed $\cA$-$\cB$ correspondence, then the representation of $\cA$ restricts to a faithful representation of $\cB$ since $\ip{\xi, b \xi} = b$ for $b \in \cB$.
	Hence, $\pi(\cA) \supseteq \cB$ can be viewed as a $\cB$-valued probability space and $\cH$ as a $\pi(\cA)$-$\cB$-correspondence.
	In \S \ref{sec: convolution sums}, we will simply assume a pointed $\mathcal{B}$-$\mathcal{B}$ correspondence $(\mathcal{H},\xi)$ is given and that $\cA$ is a $\mathrm{C}^*$-subalgebra of $\mathcal{L}(\cH_{\cB})$ containing $\cB$.
	
	\subsection{\texorpdfstring{$\cB$}{\unicodecurlyB}-valued laws and analytic transforms}
	
	In order to discuss free convolution powers, we must first describe the analog of probability distributions in the $\cB$-valued setting.  These are an analog of \emph{compactly supported} measures on $\mathbb{R}$ (the $\cB$-valued analog of measures with unbounded support is beset with technical issues).  A measure with compact support on $\mathbb{R}$ is uniquely determined by its moments (i.e.\ the expectation of polynomials), and thus can be viewed as a positive linear functional $\mu$ on the polynomial algebra $\C[x]$ satisfying $|\mu[x^k]| \leq R^k$ for some $R$.  The description in the $\cB$-valued setting is analogous, with $\C[x]$ replaced by $\cB$-valued polynomials.
	
	\begin{definition}
		Let $\mathcal{B}$ be a unital $\mathrm{C}^*$-algebra.  We define the \emph{non-commutative polynomial algebra} $\mathcal{B}\ip{x}$ to be the universal unital $*$-algebra generated by $\mathcal{B}$ and a self-adjoint indeterminate $x$.  As a vector space, $\mathcal{B}\ip{x}$ is spanned by the non-commutative monomials $b_0 x b_1 \dots x b_k$ for $k \geq 0$ and $b_j \in \mathcal{B}$.  Note that $\mathcal{B} \subseteq \mathcal{B}\ip{x}$ as unital $*$-algebras and in particular $\mathcal{B}\ip{x}$ is a $\mathcal{B}$-$\mathcal{B}$-bimodule.
	\end{definition}
	
	\begin{definition}[{$\cB$-valued laws, as in \cite{Voiculescu1995,PV2013,AW2016}}]
		Let $(\cA,E)$ be a $\cB$-valued probability space and let $X \in \cA$ be self-adjoint.  The \emph{$\cB$-valued law of $X$} is the map $\mu_X: \cB\ip{x} \to \cB$ given by $f \mapsto E[f(X)]$.  In particular, if $\cA \subseteq \mathcal{L}(\cH_{\cB})$ for some pointed $\cB$-$\cB$-correspondence $(\cH,\xi)$, then $\mu[f] = \ip{\xi, f(X)\xi}$.
	\end{definition}
	
	Just as in the case of probability measures on $\R$, there is an abstract characterization of the linear functionals $\mu: \cB\ip{x} \to \cB$ that arise as the $\cB$-valued law of some self-adjoint element with $\norm{X} \leq R$, namely, $\mu$ is a completely positive $\cB$-bimodule map with $\mu|_{\cB} = \id_\cB$ and $\norm{\mu(b_0xb_1 \dots xb_k)} \leq R^k \norm{b_0} \dots \norm{b_k}$ for all $k \geq 0$ and $b_0$, \dots, $b_k \in \cB$.  For details, see \cite[Proposition 1.2]{PV2013}, \cite[Proposition 2.9]{Williams2017}, and \cite[Theorem 2.17]{JekelLiu2020}.
	
	We also have $\cB$-valued versions of the Cauchy--Stieltjes transform and other analytic functions associated to a probability measure. For a probability measure $\mu$ on $\R$ and $\im(z) > 0$, the Cauchy--Stieltjes transform is given by
	\[
	G_\mu(z) = \int_{\R} \frac{1}{z - x}\,d\mu(x).
	\]
	The definition in the $\cB$-valued setting is similar.  Let $\mu$ be the $\cB$-valued law of some self-adjoint $X$ in a $\cB$-valued probability space $(\cA,E)$.  Then we want to define for certain $z \in \cB$,
	\[
	G_\mu(z) = E[(z - X)^{-1}].
	\]
	For this to make sense, $z - X$ must be invertible in $\cA$.  But in fact, if we denote $\im(z) = (z - z^*) / 2i$, then $\im(z) \geq \epsilon 1$ implies that $z$ is invertible with $\norm{z^{-1}} \leq 1/ \epsilon$.  And so since $X$ is self-adjoint, i.e.\ $\im(X) = 0$, we also obtain $\norm{(z - X)^{-1}} \leq 1/ \epsilon$.
	
	There is one more wrinkle in the definition of $G_\mu$ in the $\cB$-valued setting.  Namely, knowing the values of $G_\mu(z)$ is not sufficient to determine the $\cB$-valued law $\mu$ uniquely.  Indeed, looking at the power series of $G_\mu(z)$ ``at infinity'' we would only obtain moments of the form $\mu[z^{-1}xz^{-1} \dots xz^{-1}]$, but for the $\cB$-valued law, we need all moments of the form $\mu[b_0xb_1 \dots xb_n]$.  This problem can be solved by studying matrix amplifications.
	
	\begin{definition}
		Let $\mu$ be the $\cB$-valued law of some self-adjoint $X$ in a $\cB$-valued probability space $(\cA,E)$.  Let
		\[
		\mathbb{H}^{(n)}(\cB) = \{z \in M_n(\cB): \im(z) \geq \epsilon 1 \text{ for some } \epsilon > 0\}.
		\]
		Let $G_\mu^{(n)}: \mathbb{H}^{(n)}(\cB) \to M_n(\cB)$ be given by
		\[
		G_\mu^{(n)}(z) = E^{(n)}[(z - X^{(n)})^{-1}],
		\]
		where $X^{(n)}$ denotes image of $X$ in $M_n(\cA)$ under the canonical mapping $\cA \to M_n(\cA)$, that is, $X^{(n)} = X \otimes I_n$.
	\end{definition}
	
	Note that the function
	\[
	\tilde{G}_\mu^{(n)}(z) := G_\mu^{(n)}(z^{-1}) = E^{(n)}[(z^{-1} - X^{(n)})^{-1}] = E^{(n)}[z(1 - X^{(n)}z)^{-1}]
	\]
	extends to be defined in the neighborhood $B_{1/R}(0)$ in $M_n(\cB)$ where $R = \norm{X}$.  The moments of $\mu$ can then be recovered from $\tilde{G}_\mu^{(n)}$ as follows:  letting
	\[
	z = \begin{bmatrix} 0 & b_0 & 0 & \dots & 0 & 0 \\
		0 & 0 & b_1 & \dots & 0 & 0 \\
		\vdots & \vdots & \vdots & \ddots & \vdots & \vdots \\
		0 & 0 & 0 & \dots & 0 & b_n \\
		0 & 0 & 0 & \dots & 0 & 0
	\end{bmatrix} \in M_{n+2}(\cB),
	\]
	we have
	\[
	[\tilde{G}_\mu^{(n+2)}(z)]_{1,n+2} = \mu[b_0xb_1 \dots xb_n].
	\]
	
	The study of the $\cB$-valued Cauchy transform motivates the definition of $\cB$-valued non-commutative functions or $\cB$-valued fully matricial functions, which are a $\cB$-valued analog of complex-analytic functions.
	For background on these fully matricial (a.k.a. non-commutative) functions, see \cite{Voiculescu2004}, \cite{KVV2014}, \cite[\S 3]{JekelThesis}.
	
	The appropriate domain for these functions are so-called fully matricial domains:
	
	\begin{definition}
		Let $\cB$ be a $\mathrm{C}^*$-algebra.  A \emph{fully matricial domain} is a sequence of sets $\Omega^{(n)} \subseteq M_n(\cB)$ such that the following conditions hold:
		\begin{enumerate}[(1)]
			\item If $z \in \Omega^{(n)}$ and $w \in \Omega^{(m)}$, then $z \oplus w \in \Omega^{(n+m)}$.
			\item Let $z \in \Omega^{(n)}$.  For each $m$, let $z^{(m)}$ be the direct sum of $m$ copies of $z$.  Then there is some $\delta > 0$ such that $B_\delta(z^{(m)}) \subseteq \Omega^{(nm)}$ for all $m \geq 1$.
		\end{enumerate}
	\end{definition}
	
	For example, it is immediate to check that $\mathbb{H}^{(n)}(\cB)$ is a fully matricial domain, and so is the sequence $\Omega^{(n)}$ given by the ball $B_\delta(0)$ in $M_n(\cB)$.
	
	\begin{definition} \label{def: matricial function}
		Let $\Omega = (\Omega^{(n)})_{n \in \N}$ be a fully matricial domain over $\cB$.  A \emph{fully matricial function} on $\Omega$ is a sequence of functions $F^{(n)}: \Omega^{(n)} \to M_n(\cB)$ satisfying the following conditions:
		\begin{enumerate}[(1)]
			\item If $z \in \Omega^{(n)}$ and $w \in \Omega^{(m)}$, then $F^{(m+n)}(z \oplus w) = F^{(n)}(z) \oplus F^{(m)}(w)$.
			\item If $z \in \Omega^{(n)}$ and $S \in M_n(\C)$ is invertible and $S z S^{-1} \in \Omega^{(n)}$, then we have $F^{(n)}(SzS^{-1}) = S F^{(n)}(z) S^{-1}$.
			\item Given $z \in \Omega^{(n)}$, there exists some $\delta > 0$ and $M > 0$ such that for all $m \in \N$, for all $w \in B_\delta(z^{(m)})$, we have $\norm{F^{(nm)}(w)} \leq M$.
		\end{enumerate}
	\end{definition}
	
	One can check easily that the Cauchy transform $G_\mu(z)$ is a fully matricial function. Interestingly, the conditions (1) and (2) in Definition \ref{def: matricial function} are purely algebraic and (3) is merely a local boundedness condition, but these conditions automatically imply the existence of local power series expansions.  The inverse function theorem for fully matricial functions can be used to show that for a $\cB$-valued law $\mu$, the transform $\tilde{G}_\mu(z)$ is invertible in a neighborhood of $0$, and thus to define the $\cB$-valued analog of the $R$-transform, which plays a central role in the study of free convolution.  Here we single out the fact that we need for later use.
	
	\begin{proposition}[{\cite[Proposition 4.1]{Dykema2006}, \cite[Proof of Corollary 5.4]{PV2013}, \cite[Lemma 4.5.9]{JekelThesis}}] \label{prop: inversion}
		Let $\mu$ be the $\cB$-valued law of some self-adjoint operator $X$ in a $\cB$-valued probability space $(\cA,E)$ with $\norm{X} \leq r$.
		Then $\tilde{G}_\mu^{(n)}$ has a unique inverse function defined on $B_{c/r}(0)$ in $M_n(\cB)$ where $c$ is a universal constant (we can take $c = 3 - 2 \sqrt{2}$).
		The inverse function is fully matricial.
		Moreover, given $\delta > 0$, $\mu$ is uniquely determined by the values of $(\tilde{G}_\mu^{(n)})^{-1}(z)$ for $z \in B_\delta(0)$ of $M_n(\cB)$ and $n \in \N$.
	\end{proposition}
	
	\begin{definition}[{\cite[\S 4.8]{Voiculescu1995}}]
		Let $\mu$ be a $\cB$-valued law.
		Then the \emph{$R$-transform} of $\mu$ is the function $R_\mu^{(n)}(z) = (G_\mu^{(n)})^{-1}(z) - z^{-1}$ defined in a ball around zero in $M_n(\cB)$.
	\end{definition}
	
	\subsection{\texorpdfstring{$\cB$}{\unicodecurlyB}-valued free independence and free products}
	
	\begin{definition}[Free independence]
		Let $(\cA,E)$ be a $\cB$-valued probability space.  Then unital subalgebras $\cA_1$, \dots, $\cA_N$ containing $\cB$ are said to be \emph{freely independent (over $\cB$)} if we have
		\[
		E[a_1 \dots a_k] = 0
		\]
		whenever $a_j \in \cA_{i(j)}$ with $E[a_j] = 0$, provided that all consecutive indices $i(j)$ and $i(j+1)$ are distinct.  Moreover, if $X_1$, \dots, $X_N$ are self-adjoint elements of $\cA$, we say that $X_1$, \dots, $X_N$ are freely independent (over $\cB$) if the subalgebras $\cA_j$ generated by $\cB$ and $X_j$ are freely independent (over $\cB$).
	\end{definition}
	
	The following fundamental result in free probability relates the $R$-transform and free independence.
	
	\begin{proposition}[{\cite[\S 4.11]{Voiculescu1995}, \cite[Lemma 4.5]{Dykema2006}}]
		Suppose that $X$ and $Y$ are freely independent.  For ease of notation, let $R_X$ denote the $R$-transform of the law of $X$.  Then there is some $r > 0$ such that $R_{X+Y}^{(n)}(z) = R_{X}^{(n)}(z) + R_{Y}^{(n)}(z)$ for $z$ in $B_r(0)$.
		In particular, the law $\mu_{X+Y}$ is uniquely determined by $\mu_X$ and $\mu_Y$.
	\end{proposition}
	
	In this case, the law of $X+Y$ is said to be the \emph{$\cB$-valued free convolution} of $\mu_X$ and $\mu_Y$, or $\mu_{X+Y} = \mu_X \boxplus \mu_Y$.
	Since the law of $X + Y$ is uniquely determined, this construction can be used to define the free convolution $\mu \boxplus \nu$ for any two given $\cB$-valued laws $\mu$ and $\nu$, provided we can show that there exists some $\cB$-valued non-commutative probability space $(\cA,E)$ and some $X$, $Y \in \cA$ that are freely independent and realize the laws $\mu$ and $\nu$ respectively.
	This follows from the construction of $\cB$-valued free products that we will now describe.
	
	Suppose we are given $\cB$-valued probability spaces $(\cA_j,E_j)$ for $j = 1$, \dots, $N$.  Assume, as at the end of \S \ref{sec: B prob space}, that $\cA_j \subseteq \mathcal{L}((\cH_j)_{\cB})$ for some pointed $\cA_j$-$\cB$-correspondence $(\cH_j,\xi_j)$.  Then we want to construct some pointed $\cB$-$\cB$-correspondence $(\cH,\xi)$ and $*$-homomorphisms $\rho_j: \mathcal{L}((\cH_j)_{\cB}) \to \mathcal{L}(\cH_{\cB})$ such that $\rho_j(\mathcal{L}((\cH_j)_{\cB})$ are freely independent over $\cB$.  Two ingredients are needed for this construction: first, decomposing $\cH_j$ into the span of $\xi_j$ and its orthogonal complement, and second, tensor products of correspondences.
	
	Let $\cB \subseteq \cA$ and $\mathcal{H}$ be an $\mathcal{A}$-$\mathcal{B}$-correspondence; in general, if $\mathcal{K} \subseteq \mathcal{H}$ is an $\mathcal{A}$-$\mathcal{B}$-submodule, then the orthogonal complement $\mathcal{K}^\perp = \{h: \ip{h,k} = 0 \text{ for all } k \in \mathcal{K}\}$ is also an $\mathcal{A}$-$\mathcal{B}$-correspondence, but $\mathcal{K} + \mathcal{K}^\perp$ might not span all of $\mathcal{H}$ (this is related to the fact that adjointability is not automatic for bounded right $\cB$-module maps).
	However, we do have such a decomposition in the special case where $\mathcal{K}$ is the span of a $\mathcal{B}$-central unit vector in a $\mathcal{B}$-$\mathcal{B}$-correspondence.
	The following lemma is proved as in \cite[Proof of Remark 3.3]{PV2013} or \cite[Lemma 2.5.7]{JekelThesis}.
	
	\begin{lemma} \label{lem:orthocomplement}
		Let $\mathcal{H}$ be a $\mathcal{B}$-$\mathcal{B}$-correspondence and $\xi$ a $\mathcal{B}$-central unit vector.  Let $\mathcal{H}^\circ = \{h \in \mathcal{H}: \ip{h,\xi} = 0\}$.  Then we have $\mathcal{H} = \mathcal{B} \xi \oplus \mathcal{H}^\circ$ as $\mathcal{B}$-$\mathcal{B}$-correspondences.
	\end{lemma}
	
	We will also need the following notion of \emph{tensor products of $\mathrm{C}^*$-correspondences}.
	
	\begin{construction} \label{const:bimoduletensorproduct}
		Let $\mathcal{B}_0$, \dots, $\mathcal{B}_n$ be unital $\mathrm{C}^*$-algebras, and suppose that $\mathcal{H}_j$ is a Hilbert $\mathcal{B}_{j-1}$-$\mathcal{B}_j$-correspondence for each $j$.  We can form the algebraic tensor product
		\[
		\mathcal{H}_1 \otimes_{\alg, \mathcal{B}_1} \dots \otimes_{\alg, \mathcal{B}_{n-1}} \mathcal{H}_n
		\]
		in the sense of algebraic bimodules.  We define a $\cB_n$-valued semi-inner product by
		\[
		\ip{h_1 \otimes \dots \otimes h_n, h_1' \otimes \dots \otimes h_n'} = \ip{h_n, \ip{h_{n-1}, \ip{\dots, \ip{h_1, h_1'} \dots} h_{n-1}'} h_n'}.
		\]
		In other words, we first evaluate $\ip{h_1,h_1'} \in \mathcal{B}_1$, then evaluate $\ip{h_1,h_1'} h_2'$ using the left $\mathcal{B}_1$-module structure on $\mathcal{H}_2$, then compute $\ip{h_2, \ip{h_1,h_1'} h_2'} \in \mathcal{B}_2$ and so forth.  Positivity of the inner product is checked by using complete positivity in the standard way (see e.g.\ \cite[Proposition 4.5]{Lance1995}, \cite[\S 2.3.2]{JekelThesis}).
		
		We denote the separation-completion with respect to this inner product by
		\[
		\mathcal{H}_1 \otimes_{\mathcal{B}_1} \dots \otimes_{\mathcal{B}_{n-1}} \mathcal{H}_n
		\]
		and one can show that this is a Hilbert $\mathcal{B}_0$-$\mathcal{B}_n$-bimodule in the obvious way.  As one would expect, these tensor products satisfy the associativity up to a canonical isomorphism, and they distribute over direct sums of correspondences in each argument.
	\end{construction}
	
	We now come to the construction of free products of pointed $\cB$-$\cB$-correspondences as defined in \cite[\S 5]{Voiculescu1985}.  Variants of the construction were also explained \cite[\S 2-3]{Voiculescu1995} and \cite[\S 3.3, \S 3.5]{Speicher1998}.  We follow similar terminology as in \cite[\S 5]{JekelThesis}.  Suppose that $(\mathcal{H}_1,\xi_1)$, \dots, $(\mathcal{H}_N, \xi_N)$ are pointed $\cB$-$\cB$-correspondences.  Define
	\[
	\cH = *_{j=1}^N (\mathcal{H}_j, \xi_j) := \cB \xi \oplus \bigoplus_{k \geq 1} \bigoplus_{\substack{j_1, \dots, j_k \in [N] \\ j_r \neq j_{r+1}}} \cH_{j_1}^\circ \otimes_{\cB} \dots \otimes_{\cB} \cH_{j_k}^\circ.
	\]
	Here $\cB \xi$ represents a copy of $\cB$ as a $\cB$-$\cB$-correspondence with $\xi$ corresponding to the vector $1$.
	
	Next, we define $*$-homomorphisms $\rho_j: B(\cH_j) \to B(\cH)$ as follows.  For each $j$, there is a natural decomposition of $\cH$ as
	\[
	\cH \cong \cH_j \otimes_{\cB} \cM_j,
	\]
	where
	\[
	\cM_j = \cB \oplus \bigoplus_{k \geq 1} \bigoplus_{\substack{j_1, \dots, j_k \in [N] \\ j_r \neq j_{r+1} \\ j_1 \neq j}} \cH_{j_1}^\circ \otimes_{\cB} \dots \otimes_{\cB} \cH_{j_k}^\circ.
	\]
	In each case, if we denote by $U_j$ the (unitary) isomorphism $\cH \to \cH_j \otimes_{\cB} \cM_j$, we define the $*$-homomorphism $\rho_j$ by
	\[
	\rho_j(a) = U_j^*(a \otimes \id_{\cM_j}) U_j.
	\]
	Note that $\rho_j$ is injective because $\rho_j(a)$ restricted to the direct summands $\cB \oplus \cH_j^\circ$ in $\cH$ is $a$ itself conjugated by the obvious isomorphism $\cB \oplus \cH_j^\circ \to \cH_j$.  Moreover, by the same token
	\[
	\ip{\xi, \rho_j(a) \xi} = \ip{\xi_j, a \xi_j},
	\]
	which means that $\rho_j$ is expectation-preserving.  It is also easy to check that $\rho_j$ is a $\cB$-$\cB$-bimodule map, where $B(\cH_j)$ and $B(\cH)$ are given a $\cB$-$\cB$-bimodule structure through the embeddings $\cB \to B(\cH_j)$ and $\cB \to B(\cH)$ given by the left $\cB$-module structure of $\cH_j$ and $\cH$.
	
	\begin{theorem}[{See e.g.\ \cite[\S 5.3.2]{JekelThesis}}] \label{thm:productspaces}
		Let $(\cH_j,\xi_j)$ for $j = 1, \dots, N$ be $\cB$-$\cB$ correspondences with $\cB$-central unit vectors.  Let $(\cH, \xi) = *_{j=1}^N (\cH_j,\xi_j)$.  Let $E: B(\cH) \to \cB$ be the expectation given by the vector $\xi$.  Then the algebras $\rho_1(B(\cH_1))$, \dots, $\rho_N(B(\cH_N))$ are freely independent in the $\cB$-valued probability space $(B(\cH), E)$.
	\end{theorem}
	
	The construction of the free product of correspondences also naturally produces conditional expectations from $\mathcal{L}(\cH_{\cB})$ to $\mathcal{L}((\cH_j)_{\cB})$ as follows.
	
	\begin{proposition} \label{prop: conditional expectation}
		Let $(\cH,\xi)$ be the $\cB$-valued free product of pointed $\cB$-$\cB$-correspondences $(\cH_j,\xi_j)$.  Let $V: \cH_j = \cB \xi_j \oplus \cH_j^{\circ} \to \cH$ be the natural inclusion mapping $\cH_j$ onto the summands $\cB \xi$ and $\cH_j^{\circ}$ in $\cH$.
		\begin{enumerate}[(1)]
			\item Let $\mathcal{E}_j: \mathcal{L}(\cH_{\cB}) \to \mathcal{L}((\cH_j)_{\cB})$ be given by $T \mapsto V^*TV$.  Then $\mathcal{E}_j$ is a conditional expectation (with the identification of $\cA_j$ with $\rho_j(\cA_j)$).
			\item Suppose that $T = \rho_{i(1)}(a_1) \dots \rho_{i(k)}(a_k)$ where $k \geq 1$, $a_j \in \mathcal{L}((\cH_{i(j)})_{\cB})$, $i(1) \neq i(2) \neq \dots \neq i(k)$, and $\ip{\xi_\ell, a_\ell \xi_\ell} = 0$ for each $\ell$.
			Then $\mathcal{E}_j[T] = 0$ unless $k = 1$ and $i(1) = j$.
			\item Given unital $\mathrm{C}^*$-algebras $\cA_j$ with $\cB \subseteq \cA_j \subseteq \mathcal{L}((\mathcal{H}_j)_{\cB})$, $\mathcal{E}_j$ maps $\mathrm{C}^*(\rho_1(\cA_1),\dots,\rho_N(\cA_N))$ into $\cA_j$.
		\end{enumerate}
	\end{proposition}
	
	\begin{proof}
		(1) follows by direct computation.
		
		(2) First, suppose that $i(k) \neq j$.  Then by construction of the free product space, we have
		\[
		\rho_{i(1)}(a_1) \dots \rho_{i(k)}(a_k)[\cB \xi \oplus \mathcal{H}_j^\circ] \subseteq \cH_{i(1)}^\circ \otimes_{\cB} \dots \otimes_{\cB} \cH_{i(j)}^\circ \otimes (\cB \xi \oplus \cH_j^{\circ}),
		\]
		and hence the image of $\rho_{i(1)}(a_1) \dots \rho_{i(k)}(a_k) V$ is orthogonal to $\cB \xi \oplus \cH_j^{\circ}$, and so $V^* \rho_{i(1)}(a_1) \dots \rho_{i(k)}(a_k) V = 0$.  Thus, the result holds in this case.
		
		Next, suppose that $i(k) = j$.
		If $k = 1$, then there is nothing to prove.
		So suppose $k > 1$.
		Let $T' = \rho_{i(1)}(a_1) \dots \rho_{i(k-1)}(a_{k-1})$.
		Then $\mathcal{E}_j[T] = \mathcal{E}_j[T'] \rho_j(a_k)$ by the bimodule property for $\mathcal{E}_j$.
		Then we note that $\mathcal{E}_j[T'] = 0$ by the same reasoning as above.
		
		(3) We recall that the $*$-algebra generated by $\rho_j(\cA_j)$ for $j = 1$, \dots $N$ is \emph{spanned} by elements of the form $\rho_{i(1)}(a_1) \dots \rho_{i(k)}(a_k)$ where $k \geq 1$ and $a_j \in \mathcal{L}((\cH_{i(j)})_{\cB})$ and $i(1) \neq i(2) \neq \dots \neq i(k)$ and $\ip{\xi_j, a_j \xi_j} = 0$.  (One proves that arbitrary products of elements from $\cA_j$ are in the span of these reduced products, by induction on the length of the product; see the proof of \cite[Proposition 1.3]{Voiculescu1995} or \cite[Lemma 5.2.8]{JekelThesis}.)  By the previous step, $\mathcal{E}_j[\rho_{i(1)}(a_1) \dots \rho_{i(k)}(a_k)]$ is equal to zero unless $k = 1$ and $i(1) = j$, in which case it is $\mathcal{E}[\rho_j(a_1)] = a_1$.  In either case, the image is in $\cA_j$.  Hence, $\mathcal{E}_j$ maps the $*$-algebra generated by $\rho_1(\cA_1)$, \dots, $\rho_N(\cA_N)$ into $\cA_j$.  By density/continuity, it maps the $\mathrm{C}^*$-algebra generated by $\rho_1(\cA_1)$, \dots, $\rho_N(\cA_N)$ into $\cA_j$.
	\end{proof}
	
	\section{Realization of operator-valued convolution powers}
	\label{sec:opconv}
	
	For $\cB$-valued laws $\mu$ and $\nu$, and $\eta: \cB \to \cB$ completely positive, we say that $\nu$ is the $\eta$-free convolution power of $\mu$, or $\nu = \mu^{\boxplus \eta}$, if we have $R_\nu^{(n)} = \eta^{(n)} \circ R_\mu^{(n)}$ on $B_\delta(0)$ in $M_n(\cB)$ for all $n$, for some $\delta > 0$.
	In this case, by Proposition \ref{prop: inversion}, $\nu$ is uniquely determined by $\mu$ and $\eta$.
	Given $\mu$ and $\eta$, the $\eta$ free convolution power of $\mu$ does not always exist in general, but it does always exist if $\eta - \id$ is completely positive \cite{ABFN2013,Shlyakhtenko2013}.
	In fact, when $\eta - \id$ is completely positive, the following theorem describes a general construction to realize the $\eta$ free convolution power of $\mu$.
	
	\begin{theorem} \label{thm: realization}
		Let $(\cA,E)$ be a $\cB$-valued probability space.  Let $\eta: \cB \to \cB$ be a linear map such that $\eta - \id$ is completely positive.  Assume that $X$ and $V$ are freely independent over $\cB$, $X$ is a self-adjoint operator with $\cB$-valued distribution $\mu$, and $V$ satisfies
		\begin{align}
			Vb_1V^*b_2V &= V b_1 \eta(b_2) && \text{ for } b_1, b_2 \in \cB \label{eq: V condition 1} \\
			E[VbV^*] &= b && \text{ for } b \in \cB. \label{eq: V condition 2}
		\end{align}
		Consider the map $\tilde{E}[T] = E[VTV^*]$ for $T \in \cA$ (which is a conditional expectation by \eqref{eq: V condition 2} and Remark \ref{rem: conditional expectation automatic}).
		Let $\nu$ be the $\cB$-valued law of $V^*XV$ with respect to the expectation $\tilde{E}$.  Then $\nu = \mu^{\boxplus\eta}$.
	\end{theorem}
	
	In the case $\cB = \C$, condition \eqref{eq: V condition 1} is equivalent to $\eta(1)^{-1/2} V$ being a partial isometry (and in general, if $\eta(1) \in \C$, then \eqref{eq: V condition 1} implies that $\eta(1)^{-1/2} V$ is a partial isometry).
	In particular, when $\cB = \C$, we can take $V = P / E[P]^{1/2}$ for a projection $P$ and correspondingly $\eta(1) = E[P]^{-1}$.
	This leads to $\tilde{E}[VTV^*] = E[PTP] / E[P]$, or $\tilde{E}$ is the natural state on the corner $P\mathcal{A}P$.
	We thus recover the construction of free convolution powers in the scalar case from \cite{NS1996}.
	
	For general $\cB$ and $\eta \geq \id$, Shlyakhtenko \cite{Shlyakhtenko2013} gave a free Fock space construction of an operator $V$ that satisfies the stronger condition $V^*bV = \eta(b)$ for $b \in \cB$ in place of \eqref{eq: V condition 1}, and also satisfies \eqref{eq: V condition 2}.
	This stronger condition unfortunately does not permit $V$ to be a scalar multiple of a projection in the case $\cB = \C$, since if $V = \lambda P$, then $\lambda^2 P = V^* 1 V = \eta(1) \in \C$, which forces $P = 0$ or $1$.
	Besides, an operator $V$ satisfying the weaker \eqref{eq: V condition 1} and \eqref{eq: V condition 2} can be constructed in a simpler way as follows.
	
	\begin{lemma}
		\label{lem:amodel}
		Let $\eta: \cB \to \cB$ such that $\eta - \id$ is completely positive.  Then there exists a $\cB$-valued probability space $(\cA,E)$ and an operator $V$ satisfying \eqref{eq: V condition 1} and \eqref{eq: V condition 2}.
	\end{lemma}
	
	\begin{proof}
		Let $\psi = \eta - \id$.
		Let $\cH^\circ = \cB \otimes_\psi \cB$ be the $\cB$-$\cB$-correspondence associated to $\psi$ from Lemma \ref{lem: tensor product over CP map}, and let $\zeta$ be the vector $1 \otimes 1$ in $\cH^\circ$.
		Let $\cH = \cB \oplus \cH^\circ$, where $\cB$ is viewed as the trivial $\cB$-$\cB$ correspondence.
		Let $\cA$ be the space of right $\cB$-linear and adjointable operators on $\cH$, and let $E: \cA \to \cB$ be given by $E[T] = \ip{1 \oplus 0,T(1 \oplus 0)}$.  Let $V \in \cA$ be given by
		\[
		V(b \oplus h) = b \oplus \zeta b.
		\]
		It is straightforward to check that $V$ is adjointable with
		\[
		V^*(b \oplus h) = (b + \ip{\zeta,h}) \oplus 0.
		\]
		To check \eqref{eq: V condition 1}, note that for $b_1, b_2, b \in \cB$ and $h \in \cH^\circ$,
		\begin{align*}
			Vb_1 V^*b_2 V(b \oplus h) &= Vb_1 V^*b_2(b \oplus \zeta b) \\
			&= Vb_1V^*(b_2b \oplus b_2 \zeta b) \\
			&= Vb_1((b_2b + \ip{\zeta,b_2\zeta b}) \oplus 0) \\
			&= Vb_1((b_2 + \ip{\zeta,b_2\zeta})b \oplus 0) \\
			&= Vb_1 (\eta(b_2)b \oplus 0) \\
			&= Vb_1 \eta(b_2)(b \oplus h),
		\end{align*}
		where the last line follows because
		\[
		Vb_1 \eta(b_2)(0 \oplus h) = V(0 \oplus b_1 \eta(b_2)h) = 0.
		\]
		Thus, \eqref{eq: V condition 1} holds.  Moreover, since $V^*(1 \oplus 0) = (1 \oplus 0)$, we get that $E[VTV^*] = E[T]$ for $T \in \cA$, and in particular, $E[VbV^*] = b$ for $b \in \cB$, so \eqref{eq: V condition 2} holds.
	\end{proof}
	
	\begin{remark}
		Conversely, the condition that $\eta \geq \id$ is necessary in order for there to exist an operator $V$ satisfying \eqref{eq: V condition 1} and \eqref{eq: V condition 2}.  Indeed, if such a $V$ exists, then for $b \in \cB$,
		\[
		\eta(b^*b) = E[V \eta(b^*b) V^*]
		= E[VV^*b^*bVV^*] = E[(bVV^*)^*(bVV^*)] \geq E[bVV^*]^* E[bVV^*]
		= b^*b.
		\]
		The same equation also holds for $b \in M_n(\cB)$ and $V^{(n)} = V \otimes I_n \in M_n(\cB)$, and so $\eta - \id$ is completely positive.  Hence, although the proof of Theorem \ref{thm: realization} below does not use directly use that $\eta \geq \id$, the statement is vacuous outside this case.
	\end{remark}
	
	\begin{remark} \label{rem: lack of traciality}
		It is also natural to ask when the operator $V$ can be realized in a \emph{tracial} von Neumann algebra.
		Unfortunately, the situations where this can happen are rather restrictive, even with the weakened conditions \eqref{eq: V condition 1} and \eqref{eq: V condition 2}.
		Indeed, assume that $\cA$ is equipped with a faithful trace $\tau$ and that $E$ is a trace-preserving conditional expectation.
		Then for $b_1, b_2 \in \cB$, applying in order traciality, \eqref{eq: V condition 1}, traciality, \eqref{eq: V condition 2}, we have
		\begin{align*}
			\tau(V^*VV^*b_1V b_2) &= \tau(V^*b_1 V b_2 V^*V) \\
			\tau(V^*V \eta(b_1) b_2) &= \tau(V^*b_1 V b_2 \eta(1)) \\
			\tau(V \eta(b_1) b_2 V^*) &= \tau(b_1 V b_2 \eta(1) V^*) \\
			\tau(\eta(b_1) b_2) &= \tau(b_1 b_2 \eta(1)) = \tau(\eta(1) b_1 b_2),
		\end{align*}
		and since $b_2$ is arbitrary, this yields $\eta(b_1) = \eta(1) b_1$.
		By taking adjoints, $\eta(b) = b \eta(1)$ as well, and so $\eta(1)$ is in the center of $\cB$.
		Hence, the only situation where these relations can be satisfied using a trace-preserving conditional expectation is if $\eta$ is given by multiplication by an element of the center of $\cB$.
	\end{remark}
	
	Now we turn to the proof of Theorem \ref{thm: realization}.  Using a similar method to Haagerup \cite[\S 3]{Haagerup1997} and Lehner's \cite[Theorem 3.1]{Lehner2001} analytic proof that the $R$-transform is additive under free convolution, we are able to prove the result from scratch using only the definition of free independence.
	
	\begin{proof}[Proof of Theorem \ref{thm: realization}]
		Let $\nu$ be the $\cB$-valued law of $V^*XV$.
		We must show that $R_\nu^{(n)}(z) = \eta^{(n)} \circ R_\mu^{(n)}(z)$ for $z$ in $B_\delta(0)$ in $M_n(\cB)$ for all $n$, for some $\delta > 0$.
		Note that the $R$-transform being $(\tilde{G}_\mu^{(n)})^{-1}(z) - z$ can be expressed as $G_\mu^{(n)}(z^{-1} + R_\mu^{(n)}(z)) = z$ for invertible $z \in B_\delta(0)$.
		In light of analytic continuation (see e.g. \cite[Corollary 3.9.7]{JekelThesis}), since the set of invertible $z$ contains an open ball in $B_\delta(0)$, it suffices to prove that $G_\nu^{(n)}(z^{-1} + \eta^{(n)} \circ R_\mu^{(n)}(z)) = z$ for invertible $z$ in a neighborhood of $0$.
		We can also write this as
		\[
		E^{(n)}[V^{(n)}(z^{-1} + \eta^{(n)} \circ R_\mu^{(n)}(z) - (V^{(n)})^*X^{(n)} V^{(n)})^{-1}(V^{(n)})^*] = z,
		\]
		where $X^{(n)}$ and $V^{(n)}$ are the images of $X$ and $V$ under the inclusion $\cA \to M_n(\cA)$. In the following argument, we will omit the superscript $(n)$'s for ease of notation.  This should not cause confusion since the argument is a manipulation of functional and power series identities and $n$ remains fixed throughout.
		
		Now for $z$ in a neighborhood of zero,
		\begin{align*}
			V(z^{-1} + \eta \circ R_\mu(z) - V^*XV)^{-1}V^* &= V(1 - z(V^*XV - \eta \circ R_\mu(z)))^{-1} z V^* \\
			&= \sum_{k=0}^\infty V[z(V^*XV - \eta \circ R_\mu(z))]^k z V^*.
		\end{align*}
		When we expand $V[z(V^*XV - \eta \circ R_\mu(z))]^k$ by the distributive property, we get terms of the form $V$ times a product of $z V^*XV$'s and $z \eta \circ R_\mu(z)$'s.
		Applying the identity \eqref{eq: V condition 1} repeatedly as $Vz\eta\circ R_\mu(z) = VzV^*R_\mu(z)V$ beginning with the left-most occurrence of $z\eta\circ R_\mu(z)$ (which necessarily follows a $V$), we may replace all occurrences of $\eta\circ R_\mu(z)$ with $V^* R_\mu(z) V$ in the formula above.
		This yields
		\begin{align*}
			V(z^{-1} + \eta \circ R_\mu(z) - V^*XV)^{-1}V^*
			&= \sum_{k=0}^\infty V[z(V^*XV - \eta \circ R_\mu(z))]^k z V^* \\
			&= \sum_{k=0}^\infty V[z(V^*XV - V^*R_\mu(z) V]^k z V^* \\
			&= \sum_{k=0}^\infty [VzV^*(X - R_\mu(z))]^k VzV^* \\
			&= [1 - VzV^*(X - R_\mu(z))]^{-1} VzV^*.
		\end{align*}
		This is the operator whose expectation we want to equal $z$.  Now our goal is to expand this operator in terms of alternating strings of centered elements that are freely independent, which we know have expectation zero.  Thus, we introduce a certain element from $C^*(\cB,X)$ with expectation zero.  Let
		\begin{align*}
			\Phi_X(z) &= [1 - z(X - R_\mu(z))]^{-1} - 1 \\
			&= [1 - z(X - R_\mu(z))]^{-1} - [1 - z(X - R_\mu(z))][1 - z(X - R_\mu(z))]^{-1} \\
			&= z(X - R_\mu(z)) [1 - z(X - R_\mu(z))]^{-1}.
		\end{align*}
		Note that $\Phi_X(z)$ and $z^{-1} \Phi_X(z) = (X - R_\mu(z)) [1 - z(X - R_\mu(z))]^{-1}$ are well-defined in a neighborhood of $0$.
		Moreover,
		\[
		E[\Phi_X(z)] = E[(z^{-1} + R_\mu(z) - X)^{-1} z^{-1}] - 1 = G_\mu(z^{-1} + R_\mu(z)) z^{-1} - 1 = 0.
		\]
		We also note that $V^*zV - z$ has expectation zero by \eqref{eq: V condition 2}.
		
		To express $[1 - VzV^*(X - R_\mu(z))]^{-1} VzV^*$ in terms of $\Phi_X(z)$, we first write
		\begin{align*}
			1 - VzV^*(X - R_\mu(z)) &= 1 - z(X - R_\mu(z)) - (VzV^* - z)(X - R_\mu(z)) \\
			&= \left[ 1  - (VzV^* - z)(X - R_\mu(z))[1 - z(X - R_\mu(z))]^{-1} \right] [1 - z(X - R_\mu(z))] \\
			&= \left[ 1  - (VzV^* - z)z^{-1} \Phi_X(z) \right] [1 - z(X - R_\mu(z))].
		\end{align*}
		Hence,
		\begin{align*}
			[1 - VzV^*(X - R_\mu(z))]^{-1}VzV^* &= [1 - z(X - R_\mu(z))]^{-1} \left[ 1  - (VzV^* - z)z^{-1} \Phi_X(z) \right]^{-1} VzV^* \\
			&= (1 + \Phi_X(z)) \left[ 1  - (VzV^* - z)z^{-1} \Phi_X(z) \right]^{-1} [z + (VzV^* - z)].
		\end{align*}
		Then expanding the geometric series, we get
		\[
		(1 + \Phi_X(z)) \sum_{k=0}^\infty \left[(VzV^* - z)z^{-1} \Phi_X(z) \right]^k [z + (VzV^* - z)],
		\]
		or equivalently
		\begin{align}
			& \sum_{k=0}^\infty \left[(VzV^* - z)z^{-1} \Phi_X(z) \right]^k z \label{eq: final sum 1} \\
			+& \sum_{k=0}^\infty \left[(VzV^* - z)z^{-1} \Phi_X(z) \right]^k(VzV^* - z) \label{eq: final sum 2} \\
			+& \sum_{k=0}^\infty \Phi_X(z) \left[(VzV^* - z)z^{-1} \Phi_X(z) \right]^k z \label{eq: final sum 3} \\
			+& \sum_{k=0}^\infty \Phi_X(z) \left[(VzV^* - z)z^{-1} \Phi_X(z) \right]^k (VzV^* - z). \label{eq: final sum 4}
		\end{align}
		In the first sum, the $k = 0$ term is equal to $z$.  All the other terms in the four sums are alternating products of freely independent elements with expectation zero, so that freeness implies that they have expectation zero.  Hence,
		\[
		G_\nu(z^{-1} + \eta \circ R_\mu(z)) = E[[1 - VzV^*(X - R_\mu(z))]^{-1} VzV^*] = z,
		\]
		as desired.
	\end{proof}
	
	\section{Subordination and conditional expectations}
	\label{sec:condexp}
	
	Our next goal is to prove an analog of Biane's theorem on analytic subordination and conditional expectations for free additive convolution \cite[Theorem 3.1]{Biane1998}.  For motivation, we first sketch the result in the scalar-valued case; here we assume that the non-commutative probability spaces in question are tracial von Neumann algebras and we denote by $\mathrm{W}^*(X_1,\dots,X_m)$ the von Neumann subalgebra generated by $X_1$, \dots, $X_n$.  Biane showed that if $X$ and $Y$ are freely independent self-adjoint elements with distributions $\mu$ and $\nu$ respectively, then there is an analytic function $F: \H \to \H$ such that
	\[
	G_{\mu \boxplus \nu}(z) = G_\mu \circ F(z),
	\]
	but even better than that,
	\[
	E_{\mathrm{W}^*(X)}[(z - X + Y)^{-1}] = (F(z) - X)^{-1},
	\]
	where $E_{\mathrm{W}^*(X)}$ is the conditional expectation from $\mathrm{W}^*(X,Y)$ to $\mathrm{W}^*(X)$.  Furthermore, Biane stated this equation in terms of a certain Markov kernel giving ``transition probabilities from $X$ to $X + Y$.''  In the von Neumann algebraic framework, a Markov kernel is equivalent to a unital completely positive map $\Phi$ from one von Neumann algebra to another.  In this case, the mapping goes
	\[
	\Phi: L^\infty(\Spec(X+Y),\mu \boxplus \nu) \xrightarrow{\cong} \mathrm{W}^*(X+Y) \xrightarrow{\operatorname{incl}} \mathrm{W}^*(X,Y) \xrightarrow{E_{\mathrm{W}^*(X)}} \mathrm{W}^*(X) \xrightarrow{\cong} L^\infty(\Spec(X),\mu),
	\]
	where the second map is the inclusion and the third map is the conditional expectation.  Biane's result says that the function $x \mapsto (z - x)^{-1}$ in $L^\infty(\Spec(X+Y),\mu \boxplus \nu)$ is sent by $\Phi$ to the function $x \mapsto (F(z) - x)^{-1}$ in $L^\infty(\Spec(X),\mu)$.
	
	We pursue an analogous result in the setting of free convolution powers.  First, consider the scalar-valued setting.  Let $t \geq 1$, and let $P$ be a projection of trace $1/t$ freely independent of $X$.  The role of $X + Y$ in the story for additive convolution will now be played by $t PXP$ which is an element of the von Neumann algebra $P\mathrm{W}^*(X,P)P$, where the trace is given by $\tau_P(T) = \tau(PTP) / \tau(P) = t \tau(PTP)$.  We then consider the completely positive map $\Phi: L^\infty(\Spec(tPXP),\mu^{\boxplus t}) \to L^\infty(\Spec(X),\mu))$ given as follows
	\begin{multline*}
		\Phi: L^\infty(\Spec(tPXP),\mu^{\boxplus t}) \xrightarrow{\cong} \mathrm{W}^*(tPXP) \xrightarrow{\operatorname{incl}} P \mathrm{W}^*(X,P) P \\
		\xrightarrow{t \operatorname{incl}} \mathrm{W}^*(X,P) \xrightarrow{E_{\mathrm{W}^*(X)}} \mathrm{W}^*(X) \xrightarrow{\cong} L^\infty(\Spec(X),\mu).
	\end{multline*}
	The multiplication by $t$ in the middle map is necessary to make $\Phi$ unital; as further motivation, consider that $t$ times the inclusion $P\mathrm{W}^*(X,P)P \to \mathrm{W}^*(X,P)$ is the adjoint of the compression $\mathrm{W}^*(X,P) \to P\mathrm{W}^*(X,P)P$ with respect to the GNS inner products associated to the trace.  In this framework, the analog of Biane's result on subordination and conditional expectation would be that
	\[
	E_{\mathrm{W}^*(X)}[t (zP - tPXP)^{-1}] = (F(z) - X)^{-1},
	\]
	where $F$ is a function satisfying $G_{\mu^{\boxplus t}} = G_\mu \circ F$.  Here we view $P$ as the unit in $P \mathrm{W}^*(X,P) P$.  We can write this equivalently as
	\[
	E_{\mathrm{W}^*(X)}[t P(z - tPXP)^{-1}P] = (F(z) - X)^{-1}.
	\]
	
	In the more general $\cB$-valued setting of Theorem \ref{thm: realization}, we need to replace $P$ by the ``partial-isometry-like'' operator $V$.
	Moreover, conditional expectations onto $\mathrm{C}^*$-subalgebras do not exist in general, and the partial isometry $V$ often cannot be realized with respect to a conditional expectation $E$ that preserves some trace by Remark \ref{rem: lack of traciality}.
	Therefore, we will restrict our attention to a $\cB$-valued probability space $(\cA,E)$ which is the free product of $(\cA_1,E_1)$ and $(\cA_2,E_2)$ since the free product has a canonical conditional expectation $\mathcal{E}_1: \cA \to \cA_1$ described in Proposition \ref{prop: conditional expectation} (this proposition holds without assuming that either $E_1$ or $E_2$ is faithful or preserves any trace).
	The $\cB$-valued subordination/conditional expectation relation that we want is thus
	\[
	\mathcal{E}_1[V(z - V^*XV)^{-1}V^*] = (F(z) - X)^{-1}.
	\]
	
	We will prove the existence of such an $F$ using the same techniques as Theorem~\ref{thm: realization}, which again were inspired by Haagerup and Lehner's proof in the case of additive convolution (although it is also possible to give a more combinatorial proof along similar lines to Biane's argument for additive convolution \cite[Proposition 3.2]{Biane1998}).  We remark that for operator-valued convolution powers, unlike the case of additive free convolution, the existence of a subordination function $F: \H(\cB) \to \H(\cB)$ such that $G_\nu = G_\mu \circ F$ is immediate, since manipulation of functional identities shows that $F(z) = \eta^{-1}(z) + (\id - \eta^{-1})(F_\nu(z))$.  However, the conditional expectation formula is not obvious from this fact.  In any case, our proof below gives an independent argument for the existence of $F$ satisfying $G_\nu = G_\mu \circ F$.
	
	Finally, while we have presented the result for additive convolution powers in parallel with additive convolution, one could also view Theorem \ref{thm: subordination} as a subordination result for a special case of multiplicative convolution since the operator model is multiplicative in nature.  For general results on subordination for multiplicative convolution, see \cite[Theorems 3.5 and 3.6]{Biane1998}, \cite[\S 3]{BB2007subordination}, and \cite{BSTV2015multiplicative}.
	
	\begin{theorem} \label{thm: subordination}
		Let $(\cA_1,E_1)$ and $(\cA_2,E_2)$ be $\cB$-valued probability spaces and $(\cA,E)$ their free product.  Let $X \in \cA_1$ be self-adjoint and let $V \in \cA_2$ satisfy \eqref{eq: V condition 1} and \eqref{eq: V condition 2}.  Let $\mu$ be the $\cB$-valued law of $X$ and let $\nu$ be the $\cB$-valued law of $V^*XV$ with respect to the conditional expectation $\tilde{E}[T] = E[VTV^*]$.
		
		Then there exists a unique fully matricial function $F: \H(\cB) \to \H(\cB)$ such that $G_\nu = G_\mu \circ F$, and moreover $F$ satisfies
		\[
		\mathcal{E}_1[V(z - V^*XV)^{-1}V^*] = (F(z) - X)^{-1} \text{ for } z \in \H(\cB)
		\]
		where $\mathcal{E}_1: \cA \to \cA_1$ is the canonical conditional expectation from the free product construction (Proposition \ref{prop: conditional expectation}).
	\end{theorem}
	
	\begin{proof}
		As before, we omit the superscripts $(n)$ for ease of notation.  First, we show that for $z$ in a neighborhood of zero, we have
		\begin{equation} \label{eq: subordination in neighborhood of infinity}
			\mathcal{E}_1
			[V(G_\nu^{-1}(z) - V^*XV)^{-1}V^*] = (G_\mu^{-1}(z) - X)^{-1}.
		\end{equation}
		Recall that $G_\nu^{-1}(z) = z^{-1} + R_\nu(z) = z^{-1} + \eta \circ R_\mu(z)$.
		From the proof of Theorem \ref{thm: realization}, \[V(G_\nu^{-1}(z) - V^*XV)^{-1}V^* = V(z^{-1} + \eta \circ R_\mu(z) - V^*XV)^{-1} V^*\] can be expressed as the sum \eqref{eq: final sum 1} - \eqref{eq: final sum 4}.  Now the conditional expectation onto $\cA_1$ of an alternating product of centered terms from $\cA_1$ and $\cA_2$ is zero unless the product has length zero, or the product has a single term from $\cA_1$.  Therefore, the only terms in \eqref{eq: final sum 1} - \eqref{eq: final sum 4} that survive under the conditional expectation are the $k = 0$ terms in \eqref{eq: final sum 1} and \eqref{eq: final sum 3}.  Thus,
		\[
		\mathcal{E}_1[V(G_\nu^{-1}(z) - V^*XV)^{-1}V^*] = z + \Phi_X(z) z.
		\]
		Recalling the definition of $\Phi_X(z)$,
		\begin{align*}
			z + \Phi_X(z) z &= (1 + \Phi_X(z))z \\
			&= [1 - z(X - R_\mu(z))]^{-1} z \\
			&= [z^{-1} + R_\mu(z) - X]^{-1} \\
			&= (G_\mu^{-1}(z) - X)^{-1}.
		\end{align*}
		This shows \eqref{eq: subordination in neighborhood of infinity}.  By substituting $G_\nu(z)$ instead of $z$ into \eqref{eq: subordination in neighborhood of infinity}, we deduce that if $z^{-1}$ is sufficiently small, then
		\[
		\mathcal{E}_1[V(z - V^*XV)^{-1}V^*] = (G_\mu^{-1} \circ G_\nu(z) - X)^{-1}.
		\]
		This is the identity that we want with $F = G_\mu^{-1} \circ G_\nu$.
		
		Next, we must prove that $F$ extends to all of $\H(\cB)$.  Let $z \in \H(\cB)$.  Then $\im(z) \geq \epsilon 1$ for some $\epsilon > 0$.
		Now $\im (z - V^*XV) = \im(z) \geq \epsilon 1$.  Thus,
		\begin{align*}
			\im (z - V^*XV)^{-1} &= \im [(z - V^*XV)^{-1} (z^* - V^*XV) (z^* - V^*XV)^{-1}] \\
			&= (z - V^*XV)^{-1} \im (z^* - V^*XV) (z^* - V^*XV)^{-1} \\
			&= -(z - V^*XV)^{-1} \im (z) (z^* - V^*XV)^{-1} \\
			&\leq -\epsilon (z - V^*XV)^{-1}(z^* - V^*XV)^{-1} \\
			&= -\epsilon |(z - V^*XV)|^{-2} \\
			&\leq -\epsilon \norm{z - V^*XV}^{-2} 1.
		\end{align*}
		Therefore,
		\[
		\im \mathcal{E}_1[V(z - V^*XV)^{-1}V^*] \leq -\epsilon \norm{z - V^*XV}^{-2} 
		\mathcal{E}_1[VV^*].
		\]
		Since $VV^*$ is freely independent of $\cA_1$ in the free product, we get $\mathcal{E}_1[VV^*] = E[VV^*] = 1$.  Thus,
		\[
		\im \mathcal{E}_1[V(z - V^*XV)^{-1}V^*] \leq -\epsilon \norm{z - V^*XV}^{-2} \cE_1[VV^*].
		\]
		Hence, $\mathcal{E}_1[V(z - V^*XV)^{-1}V^*]$ has imaginary part less than a negative multiple of $1$, and so by similar reasoning as we used before, $\mathcal{E}_1[V(z - V^*XV)^{-1}V^*]$ is invertible in $\mathcal{L}(H_1)$ and its inverse has imaginary part bounded below by a positive multiple of the identity.  Let $\widehat{F}: \H(\cB) \to \H(\cA_1)$ be given by
		\[
		\widehat{F}(z) = (\mathcal{E}_1[V(z - V^*XV)^{-1}V^*])^{-1} + X.
		\]
		
		We claim that $\widehat{F}$ defines a fully matricial function $\mathbb{H}(\cB) \to \mathbb{H}(\cA_1)$.  Preservation of direct sums and similarity by scalar matrices in Definition \ref{def: matricial function} (1) and (2) follow by straightforward algebraic computations since multiplication by $X^{(n)}$ and $V^{(n)}$ respects direct sums and similarities.  The local boundedness condition Definition \ref{def: matricial function} (3) follows from the fact that $\widehat{F}(z)$ is uniformly bounded for all $z$ with $\im(z) \geq \epsilon$ and $\norm{z} \leq R$ as a consequence of the estimates we made in the process of showing invertibility of $z - V^*XV$ and $\mathcal{E}_1[V(z - V^*XV)^{-1}V^*]$.
		
		Now if $z$ is invertible with $\norm{z^{-1}}$ sufficiently small, we have that $\widehat{F}(z) = F(z)$, so that $\widehat{F}(z)$ actually takes values in $\mathbb{H}(\cB)$.  Hence, we have when $z^{-1}$ is sufficiently small that
		\[
		\widehat{F}(z) = E[\widehat{F}(z)].
		\]
		By the identity theorem (see e.g.\ \cite[Corollary 3.9.7]{JekelThesis}), this equality extends to all of $\H(\cB)$.  Therefore, $\widehat{F}(z)$ always maps $\H(\cB)$ into $\H(\cB)$, and hence provides the desired extension of $F(z)$ to all of $\H(\cB)$.  Moreover, since $G_\nu = G_\mu \circ F$ holds when $z \in \H(\cB)$ with $\norm{z^{-1}}$ sufficiently small, again by the identity theorem this relation extends to all of $\H(\cB)$.  Uniqueness of $F$ follows again by the identity theorem because $F(z) = G_\mu^{-1} \circ G_\nu(z)$ when $z^{-1}$ is sufficiently small.
	\end{proof}
	
	\section{On convolution powers and sums} \label{sec: convolution sums}
	
	In the case where $\eta = n \id$, the free convolution power $\mu^{\boxplus \eta}$ is simply the $n$-fold free additive convolution $\mu \boxplus \dots \boxplus \mu$.  Thus, it is natural to ask how the construction of $\mu^{\boxplus n}$ in Theorem \ref{thm: realization} relates to the construction of free additive convolution by adding copies of the original operator in a free product space.  We will now explain the relationship through an explicit transformation of the underlying $\cB$-$\cB$-correspondences (Hilbert spaces in the scalar-valued case).
	
	\begin{theorem} \label{thm: transformation}
		Let $X$ be an operator on the $\cB$-$\cB$-correspondence $\cH$ with the expectation given by a unit vector $\xi$.  Write $\cH = \cB \xi \oplus \cH^\circ$. Consider the free product correspondence $*_{j=1}^n (\cH_j,\xi_j)$ where $(\cH_j,\xi_j)$ is a copy of $(\cH,\xi)$.  Let $X_j$ be the $j$th copy of $X$ acting on the free product.
		
		Let $\cK$ be the $\cB$-$\cB$-correspondence $\cB^{\oplus n}$, and let $e_j$ be the $j$th canonical basis vector.  Let $\cK^\circ = 0 \oplus \cB^{\oplus (n-1)}$ which is the orthogonal complement of $\cB e_1$.  Let $V \in \mathcal{L}(\mathcal{K})$ be the operator 
		\[
		V(b_1 \oplus \dots \oplus b_n) = \left( \frac{b_1 + \dots + b_n}{\sqrt{n}} \right)^{\oplus n}.
		\]
		Form the free product correspondence $(\mathcal{K},e_1) * (\cH,\xi)$, and view $V$ and $X$ as operators on $(\mathcal{K},e_1) * (\cH,\xi)$.
		
		There exists an isomorphism of $\cB$-$\cB$-correspondences
		\[
		\Phi: *_{j=1}^n (\cH_j,\xi_j) \to (V(\cH * \cK), V(\xi))
		\]
		(explicitly constructed in the proof below), such that
		\begin{equation} \label{eq: intertwining relation}
			\Phi \circ (X_1 + \dots + X_n) = VXV \circ \Phi.
		\end{equation}
	\end{theorem}
	
	\begin{remark}
		Note that according to Theorem \ref{thm: realization}, if $\mu^{\boxplus n \id}$ is the $\cB$-valued law of $X$, then $\mu^{\oplus n}$ is the law of the operator $VXV$ on the pointed $\cB$-$\cB$-correspondence $V[(\cH,\xi) * (\cK,e_1)]$, where the unit vector is given by $V$ applied to the original unit vector.  Meanwhile, $\mu^{\boxplus n}$ is also the free convolution of $n$ copies of $\mu$, and so it is the $\cB$-valued law of the operator $X_1 + \dots + X_n$ on $*_{j=1}^n (\cH_j,\xi_j)$.  Thus, $VXV$ and $X_1 + \dots + X_n$ must have the same $\cB$-valued law, and Theorem \ref{thm: transformation} is asserting that this fact can be witnessed by an explicit spatial isomorphism.
	\end{remark}
	
	\begin{proof}[Proof of Theorem \ref{thm: transformation}]
		In other words, $V$ is $\sqrt{n}$ times the projection onto the $\cB \varepsilon$, where $\varepsilon = (1 \oplus \dots \oplus 1) / \sqrt{n}$.
		We denote this projection by $P = V / \sqrt{n}$.
		Note that $V(e_j) = \varepsilon$ for $j = 1$, \dots, $n$.
		
		Note that $(\cH,\xi) * (\cK,e_1)$ can be written as
		\begin{equation} \label{eq: free product expansion 1}
			(\cH,\xi) * (\cK,e_1) = \left( \bigoplus_{k=0}^\infty \cK \otimes_{\cB} (\cH^\circ \otimes_{\cB} \cK^\circ)^{\otimes_{\cB} k} \otimes_{\cB} \cH, e_1 \otimes \xi \right),
		\end{equation}
		where $e_1 \otimes \xi$ is the vector in the $k = 0$ summand $\cK \otimes \cH$.  Since $\cK^\circ = \bigoplus_{j=2}^n \cB e_j$, we can express the underlying correspondence as
		\[
		\bigoplus_{k=0}^\infty \bigoplus_{j_1,\dots,j_k = 2, \dots, n} \cK \otimes_{\cB} (\cH^\circ \otimes_{\cB} \cB e_{j_1}) \otimes_{\cB} \dots \otimes_{\cB} (\cH^\circ \otimes_{\cB} \cB e_{j_k}) \otimes_{\cB} \cH.
		\]
		Thus,
		\[
		V[(\cH,\xi) * (\cK,e_1)]
		= \left( \bigoplus_{k=0}^\infty \bigoplus_{j_1,\dots,j_k = 2, \dots, n} \cB \varepsilon \otimes_{\cB} (\cH^\circ \otimes_{\cB} \cB e_{j_1}) \otimes_{\cB} \dots \otimes_{\cB} (\cH^\circ \otimes_{\cB} \cB e_{j_k}) \otimes_{\cB} \cH, e_1 \otimes \xi \right).
		\]
		
		Meanwhile, we write
		\[
		*_{j=1}^n (\cH_j,\xi_j) = \left( \bigoplus_{k=0}^\infty \bigoplus_{j_1 \neq j_2 \neq \dots \neq j_k \neq 1} \cH_{j_1}^\circ \otimes_{\cB} \dots \otimes_{\cB} \cH_{j_k}^{\circ} \otimes_{\cB} \cH_1, \xi_1 \right),
		\]
		where $\xi_1$ is the vector in the $k = 0$ summand $\cH_1$.  Let $\Phi$ be the map given on each of the summands in $*_{j=1}^n (\cH_j,\xi_j)$ as follows
		\begin{align*}
			\cH_{j_1}^\circ \otimes_{\cB} \dots \otimes_{\cB} \cH_{j_k}^{\circ} \otimes_{\cB} \cH_1 &\cong \cH^\circ \otimes_{\cB} \dots \otimes_{\cB} \cH^\circ \otimes_{\cB} \cH \\
			&\cong \cB \otimes_{\cB} (\cH^\circ \otimes_{\cB} \cB) \otimes_{\cB} \dots \otimes_{\cB} (\cH^\circ \otimes_{\cB} \cB) \otimes_{\cB} \cH \\
			&\cong \cB \varepsilon \otimes_{\cB} (\cH^\circ \otimes_{\cB} \cB e_{j_1 - j_2 + 1}) \otimes_{\cB} \dots \otimes_{\cB} (\cH^\circ \otimes_{\cB} \cB e_{j_{k-1}-j_k +1}) \otimes_{\cB} (\cH^\circ \otimes_{\cB} \cB e_{j_k}) \otimes_{\cB} \cH,
		\end{align*}
		where the indices are considered modulo $n$.  Since the string is alternating, $j_{i-1} - j_i + 1$ will never equal $1$ modulo $n$, and this is why the term on the right-hand side is one of the summands occurring in $V[(\cH,\xi) * (\cK,e_1)]$.  Note that $\Phi(\xi_1) = \varepsilon \otimes \xi$ as desired.
		
		Now we prove the intertwining relation \eqref{eq: intertwining relation}.  Fix a vector
		\[
		\theta = h_1 \otimes \dots \otimes h_k \otimes h_0 \in \cH_{j_1}^\circ \otimes_{\cB} \dots \otimes_{\cB} \cH_{j_k}^{\circ} \otimes_{\cB} \cH_1 \subseteq *_{j=1}^n (\cH_j,\xi_j).
		\]
		Let's first consider the generic case where $k > 0$.  Viewing $h_1$ as an element of $\cH$, write
		\[
		X \xi = b \xi \oplus \zeta, \qquad Xh_1 = b' \oplus h_1'
		\]
		where $b, b' \in \cB$ and $\zeta, h_1' \in \cH^\circ$.  Then we have
		\[
		X_{j_1}[h_1 \otimes \dots \otimes h_k \otimes h_0] = b'h_2 \otimes \dots \otimes h_k \otimes h_0 + h_1' \otimes h_2 \otimes \dots \otimes h_k \otimes h_0.
		\]
		Moreover, if $j \neq j_1$, then
		\[
		X_j[h_1 \otimes \dots \otimes h_k \otimes h_0] = bh_1 \oplus h_2 \otimes \dots \otimes h_k \otimes h_0 + \zeta_j \otimes h_1 \otimes h_2 \otimes \dots \otimes h_k \otimes h_0,
		\]
		where $\zeta_j$ is the copy of $\zeta$ in $\cH_j^\circ$.  Thus,
		\begin{align*}
			(X_1 + \dots + X_n)[\theta]
			&= b'h_2 \otimes \dots \otimes h_k \otimes h_0 \\
			& \quad + h_1' \otimes h_2 \otimes \dots \otimes h_k \otimes h_0 \\
			& \quad + (n-1)bh_1 \otimes h_2 \otimes \dots \otimes h_k \otimes h_0 \\
			& \quad + \sum_{j \neq j_1} \zeta_j \otimes h_1 \otimes h_2 \otimes \dots \otimes h_k \otimes h_0.
		\end{align*}
		Then applying the map $\Phi$, we get
		\begin{align*}
			\Phi \circ (X_1 + \dots + X_n)[\theta]
			&= \varepsilon \otimes (b'h_2 \otimes e_{j_2-j_3+1}) \otimes \dots \otimes (h_k \otimes e_{j_k}) \otimes h_0 \\
			& \quad + \varepsilon \otimes (h_1' \otimes e_{j_1-j_2+1}) \otimes h_2 \otimes \dots \otimes (h_k \otimes e_{j_k}) \otimes h_0 \\
			& \quad + \varepsilon \otimes ((n-1)bh_1 \otimes e_{j_1-j_2+1})) \otimes (h_2 \otimes e_{j_2-j_3+1}) \otimes \dots \otimes (h_k \otimes e_{j_k}) \otimes h_0 \\
			& \quad + \sum_{j \neq j_1} \varepsilon \otimes (\zeta \otimes e_{j-j_1+1}) \otimes (h_1 \otimes e_{j_1-j_2+1}) \otimes (h_2 \otimes e_{j_2-j_3+1}) \otimes \dots \otimes (h_k \otimes e_{j_k}) \otimes h_0.
		\end{align*}
		Furthermore, we can reindex the sum over $j \neq j_1$ of $e_{j-j_1+1}$ as the sum over $j \neq 1$ of $e_j$.
		
		Now on the other hand, we need to compute $VXV \circ \Phi$ on our original vector.  First,
		\[
		\Phi[h_1 \otimes \dots \otimes h_k \otimes h_0] = \varepsilon \otimes (h_1 \otimes e_{j_1-j_2+1}) \otimes \dots \otimes (h_k \otimes e_{j_k}) \otimes h_0.
		\]
		The vector is already in the image of $V$, so $V \Phi[\theta] = \sqrt{n} \Phi[\theta]$.  Now we compute the application of the operator $X$, which is defined by decomposing the $\cB$-$\cB$-correspondence $(\cH,\xi) * (\cK,e_1)$ into a tensor product of $\cH$ on the left with other terms.  To this end, we write
		\begin{align*}
			\sqrt{n} \Phi[\theta] &= \sum_{j=1}^n e_j \otimes (h_1 \otimes e_{j_1-j_2+1}) \otimes \dots \otimes (h_k \otimes e_{j_k}) \otimes h_0 \\
			&= (h_1 \otimes e_{j_1-j_2+1}) \otimes \dots \otimes (h_k \otimes e_{j_k}) \otimes h_0 \\
			& \quad +
			\sum_{j=2}^n e_j \otimes (h_1 \otimes e_{j_1-j_2+1}) \otimes \dots \otimes (h_k \otimes e_{j_k}) \otimes h_0,
		\end{align*}
		where we have separated out the $e_1$ term and written it without the $e_1$ because $e_1$ is the state vector.  We apply the operator $X$ and obtain
		\begin{align*}
			XV\Phi[\theta] &= b'e_{j_1-j_2+1} \otimes (h_2 \otimes e_{j_3 - j_2+1}) \otimes \dots \otimes (h_k \otimes e_{j_k}) \otimes h_0 \\
			&\quad + (h_1' \otimes e_{j_1-j_2+1}) \otimes \dots \otimes (h_k \otimes e_{j_k}) \otimes h_0 \\
			&\quad + \sum_{j=2}^n be_j \otimes (h_1 \otimes e_{j_1-j_2+1}) \otimes \dots \otimes (h_k \otimes e_{j_k}) \otimes h_0 \\
			&\quad + \sum_{j=2}^n (\zeta \otimes e_j) \otimes (h_1 \otimes e_{j_1-j_2+1}) \otimes \dots \otimes (h_k \otimes e_{j_k}) \otimes h_0.
		\end{align*}
		When we write this again in terms of the expansion \eqref{eq: free product expansion 1}, we tack on $e_1 \otimes$ to the front of the second and fourth lines.  Then applying $V$ again,
		\begin{align*}
			VXV\Phi[\theta] &= b'V(e_{j_1-j_2+1}) \otimes (h_2 \otimes e_{j_3 - j_2+1}) \otimes \dots \otimes (h_k \otimes e_{j_k}) \otimes h_0 \\
			&\quad + V(e_1) \otimes (h_1' \otimes e_{j_1-j_2+1}) \otimes \dots \otimes (h_k \otimes e_{j_k}) \otimes h_0 \\
			&\quad + \sum_{j=2}^n bV(e_j) \otimes (h_1 \otimes e_{j_1-j_2+1}) \otimes \dots \otimes (h_k \otimes e_{j_k}) \otimes h_0 \\
			&\quad + \sum_{j=2}^n V(e_1) \otimes (\zeta \otimes e_j) \otimes (h_1 \otimes e_{j_1-j_2+1}) \otimes \dots \otimes (h_k \otimes e_{j_k}) \otimes h_0.
		\end{align*}
		Since $V(e_j) = \varepsilon$ and since $b \varepsilon \otimes \phi = \varepsilon b \otimes \phi = \varepsilon \otimes b\phi$ for vectors $\phi$, our computation of $VXV \Phi[\theta]$ agrees with our previous expression for $\Phi \circ (X_1 + \dots + X_n)[\theta]$.
		
		In the cases where $k = 0$, we can further split into the case of a vector $h_0$ from $\cH^\circ$ and a multiple of $\xi$.  These arguments proceed in a similar way, and are left as an exercise for the reader.
	\end{proof}
	
	\bibliographystyle{plainnat}
	\bibliography{convolutionpowers}

\begin{thebibliography}{31}
\providecommand{\natexlab}[1]{#1}
\providecommand{\url}[1]{\texttt{#1}}
\expandafter\ifx\csname urlstyle\endcsname\relax
  \providecommand{\doi}[1]{doi: #1}\else
  \providecommand{\doi}{doi: \begingroup \urlstyle{rm}\Url}\fi

\bibitem[Anderson et~al.(2009)Anderson, Guionnet, and Zeitouni]{AGZ2009}
Greg~W. Anderson, Alice Guionnet, and Ofer Zeitouni.
\newblock \emph{An Introduction to Random Matrices}.
\newblock Cambridge Studies in Advanced Mathematics. Cambridge University
  Press, 2009.
\newblock \doi{10.1017/CBO9780511801334}.

\bibitem[Anshelevich and Williams(2016)]{AW2016}
Michael Anshelevich and John~D. Williams.
\newblock Operator-valued monotone convolution semigroups and an extension of
  the {B}ercovici-{P}ata bijection.
\newblock \emph{Doc. Math.}, 21:\penalty0 841--871, 2016.
\newblock \doi{10.4171/dm/547}.

\bibitem[Anshelevich et~al.(2013)Anshelevich, Belinschi, F{\'e}vrier, and
  Nica]{ABFN2013}
Michael Anshelevich, Serban~T. Belinschi, Maxime F{\'e}vrier, and Alexandru
  Nica.
\newblock Convolution powers in the operator-valued framework.
\newblock \emph{Transactions of the American Mathematical Society},
  365:\penalty0 2063--2097, 2013.
\newblock \doi{10.1090/S0002-9947-2012-05736-9}.

\bibitem[Belinschi and Bercovici(2007)]{BB2007subordination}
Serban~T. Belinschi and Hari Bercovici.
\newblock A new approach to subordination results in free probability.
\newblock \emph{Journal d’Analyse Mathematique}, 101\penalty0 (1):\penalty0
  357--365, 2007.
\newblock \doi{10.1007/s11854-007-0013-1}.

\bibitem[Belinschi and Bercovici(2025)]{BB2022}
Serban~T. Belinschi and Hari Bercovici.
\newblock Upgrading subordination properties in free probability theory.
\newblock \emph{J. Operator Th.}, 93\penalty0 (1):\penalty0 115--122, 2025.
\newblock \doi{10.7900/jot.2022dec2022.2405}.

\bibitem[Belinschi et~al.(2015)Belinschi, Speicher, Treilhard, and
  Vargas]{BSTV2015multiplicative}
Serban~T. Belinschi, Roland Speicher, John Treilhard, and Carlos Vargas.
\newblock Operator-valued free multiplicative convolution: Analytic
  subordination theory and applications to random matrix theory.
\newblock \emph{International Mathematics Research Notices IMRN}, 2015\penalty0
  (14):\penalty0 5933--5958, 2015.
\newblock \doi{10.1093/imrn/rnu114}.

\bibitem[Belinschi et~al.(2017)Belinschi, Mai, and Speicher]{BMS2017}
Serban~T. Belinschi, Tobias Mai, and Roland Speicher.
\newblock Analytic subordination theory of operator-valued free additive
  convolution and the solution of a general random matrix problem.
\newblock \emph{J. Reine Angew. Math.}, 732:\penalty0 21--53, 2017.
\newblock \doi{10.1515/crelle-2014-0138}.

\bibitem[Bercovici and Pata(1999)]{BP1999}
Hari Bercovici and Vittorino Pata.
\newblock Stable laws and domains of attraction in free probability theory.
\newblock \emph{Ann. of Math. (2)}, 149\penalty0 (3):\penalty0 1023--1060,
  1999.
\newblock ISSN 0003-486X,1939-8980.
\newblock \doi{10.2307/121080}.
\newblock With an appendix by Philippe Biane.

\bibitem[Bercovici and Voiculescu()]{BV1995}
Hari Bercovici and Dan-Virgil Voiculescu.
\newblock Superconvergence to the central limit and failure of the {C}ram\'{e}r
  theorem for free random variables.
\newblock \emph{Probab. Theory Related Fields}, \penalty0 (2):\penalty0
  215--222.
\newblock ISSN 0178-8051,1432-2064.
\newblock \doi{10.1007/BF01204215}.

\bibitem[Biane(1998)]{Biane1998}
Philippe Biane.
\newblock Processes with free increments.
\newblock \emph{Mathematische Zeitschrift}, 227\penalty0 (1):\penalty0
  143--174, 1998.
\newblock \doi{10.1007/PL00004363}.

\bibitem[Blackadar(2006)]{Blackadar2006}
Bruce Blackadar.
\newblock \emph{Operator Algebras: Theory of ${C}^*$-algebras and von {N}eumann
  algebras}, volume 122 of \emph{Encyclopaedia of Mathematical Sciences}.
\newblock Springer-Verlag, Berlin, Heidelberg, 2006.
\newblock \doi{10.1007/3-540-28517-2}.

\bibitem[Dykema(2006)]{Dykema2006}
Kenneth~J. Dykema.
\newblock On the s-transform over a banach algebra.
\newblock \emph{Journal of Functional Analysis}, 231\penalty0 (1):\penalty0
  90--110, 2006.
\newblock ISSN 0022-1236.
\newblock \doi{10.1016/j.jfa.2005.04.008}.

\bibitem[Haagerup(1997)]{Haagerup1997}
Uffe Haagerup.
\newblock On {V}oiculescu's {$R$}- and {$S$}-transforms for free non-commuting
  random variables.
\newblock In Dan-Virgil Voiculescu, editor, \emph{Free Probability Theory},
  volume~12 of \emph{Fields Institute Communications}, pages 127--148. Amer.
  Math. Soc., 1997.
\newblock \doi{10.1090/fic/012}.

\bibitem[Jekel(2020)]{JekelThesis}
David Jekel.
\newblock \emph{Evolution equations in non-commutative probability}.
\newblock PhD thesis, University of California, Los Angeles, 2020.

\bibitem[Jekel and Liu(2020)]{JekelLiu2020}
David Jekel and Weihua Liu.
\newblock An operad of non-commutative independences defined by trees.
\newblock \emph{Dissertationes Mathematicae}, 553:\penalty0 1--100, 2020.
\newblock \doi{10.4064/dm797-6-2020}.

\bibitem[Kaliuzhnyi-Verbovetskyi and Vinnikov(2014)]{KVV2014}
Dmitry~S. Kaliuzhnyi-Verbovetskyi and Victor Vinnikov.
\newblock \emph{Foundations of Free Non-Commutative Function Theory}, volume
  199 of \emph{Mathematical Surveys and Monographs}.
\newblock American Mathematical Society, 2014.
\newblock ISBN 978-1-4704-1697-3.
\newblock \doi{10.1090/surv/199}.

\bibitem[Lance(1995)]{Lance1995}
E.~Christopher Lance.
\newblock \emph{{H}ilbert ${C}^*$-Modules: A Toolkit for Operator Algebraists}.
\newblock London Mathematical Society Lecture Note Series. Cambridge University
  Press, Cambridge, 1995.
\newblock \doi{10.1017/CBO9780511526206}.

\bibitem[Lehner(2001)]{Lehner2001}
Franz Lehner.
\newblock On the computation of spectra in free probability.
\newblock \emph{Journal of Functional Analysis}, 183\penalty0 (2):\penalty0
  451--471, 2001.
\newblock ISSN 0022-1236.
\newblock \doi{10.1006/jfan.2001.3748}.

\bibitem[Liu(2021)]{Liu2021}
Weihua Liu.
\newblock Relations between convolutions and transforms in operator-valued free
  probability.
\newblock \emph{Advances in Mathematics}, 390:\penalty0 107949, 2021.
\newblock ISSN 0001-8708.
\newblock \doi{10.1016/j.aim.2021.107949}.

\bibitem[Mingo and Speicher(2017)]{MS2017}
James~A. Mingo and Roland Speicher.
\newblock \emph{Free probability and random matrices}, volume~35 of
  \emph{Fields Institute Monographs}.
\newblock Springer, New York; Fields Institute for Research in Mathematical
  Sciences, Toronto, ON, 2017.
\newblock ISBN 978-1-4939-6941-8; 978-1-4939-6942-5.
\newblock \doi{10.1007/978-1-4939-6942-5}.

\bibitem[Nica and Speicher(1996)]{NS1996}
Alexandru Nica and Roland Speicher.
\newblock On the multiplication of free {$N$}-tuples of noncommutative random
  variables.
\newblock \emph{Amer. J. Math.}, 118\penalty0 (4):\penalty0 799--837, 1996.
\newblock ISSN 0002-9327,1080-6377.
\newblock \doi{10.1353/ajm.1996.0034}.

\bibitem[Paschke(1973)]{Paschke1973}
William~L. Paschke.
\newblock Inner product modules over {$B\sp{\ast} $}-algebras.
\newblock \emph{Trans. Amer. Math. Soc.}, 182:\penalty0 443--468, 1973.
\newblock \doi{10.2307/1996542}.

\bibitem[Popa and Vinnikov(2013)]{PV2013}
Mihai Popa and Victor Vinnikov.
\newblock Non-commutative functions and the non-commutative {L}{\'e}vy-{H}in{\v
  c}in formula.
\newblock \emph{Advances in Mathematics}, 236:\penalty0 131--157, 2013.
\newblock \doi{10.1016/j.aim.2012.12.013}.

\bibitem[Shlyakhtenko(2013)]{Shlyakhtenko2013}
Dimitri Shlyakhtenko.
\newblock On operator-valued free convolution powers.
\newblock \emph{Indiana University Mathematics Journal}, 62\penalty0
  (1):\penalty0 91--97, 2013.
\newblock \doi{10.1512/iumj.2013.62.4863}.

\bibitem[Speicher(1998)]{Speicher1998}
Roland Speicher.
\newblock \emph{Combinatorial theory of the free product with amalgamation and
  operator-valued free probability theory}, volume 132 of \emph{Memoirs of the
  American Mathematical Society}.
\newblock American Mathematical Soc., 1998.
\newblock \doi{10.1090/memo/0627}.

\bibitem[Voiculescu(1985)]{Voiculescu1985}
Dan-Virgil Voiculescu.
\newblock Symmetries of some reduced free product ${C}^*$-algebras.
\newblock In Huzihiro Araki, Calvin~C. Moore, {\c{S}}erban-Valentin Stratila,
  and Dan Voiculescu, editors, \emph{Operator Algebras and their Connections
  with Topology and Ergodic Theory}, pages 556--588. Springer, Berlin,
  Heidelberg, 1985.
\newblock ISBN 978-3-540-39514-0.
\newblock \doi{10.1007/BFb0074909}.

\bibitem[Voiculescu(1993)]{Voiculescu1993}
Dan-Virgil Voiculescu.
\newblock The analogues of entropy and of {F}isher's information measure in
  free probability theory, {I}.
\newblock \emph{Communications in Mathematical Physics}, 155\penalty0
  (1):\penalty0 71--92, 1993.
\newblock \doi{10.1007/bf02100050}.

\bibitem[Voiculescu(1995)]{Voiculescu1995}
Dan-Virgil Voiculescu.
\newblock Operations on certain non-commutative operator-valued random
  variables.
\newblock \emph{Ast{\'e}risque}, 232\penalty0 (1):\penalty0 243--275, 1995.
\newblock \doi{10.24033/ast.322}.

\bibitem[Voiculescu(2002)]{Voiculescu2002b}
Dan-Virgil Voiculescu.
\newblock Analytic subordination consequences of free {M}arkovianity.
\newblock \emph{Indiana University Mathematics Journal}, 51:\penalty0
  1161--1166, 2002.
\newblock \doi{10.1512/iumj.2002.51.2252}.

\bibitem[Voiculescu(2004)]{Voiculescu2004}
Dan-Virgil Voiculescu.
\newblock Free analysis questions. {I}. {D}uality transform for the coalgebra
  of {$\partial_{X\colon B}$}.
\newblock \emph{Int. Math. Res. Not.}, 2004\penalty0 (16):\penalty0 793--822,
  2004.
\newblock ISSN 1073-7928,1687-0247.
\newblock \doi{10.1155/S1073792804132443}.

\bibitem[Williams(2017)]{Williams2017}
John~D. Williams.
\newblock Analytic function theory for operator-valued free probability.
\newblock \emph{Journal f{\"u}r die reine und angewandte Mathematik (Crelles
  Journal)}, 2017:\penalty0 119--149, 2017.
\newblock \doi{10.1515/crelle-2014-0106}.

\end{thebibliography}
	
\end{document}